\documentclass[12pt]{article}
\usepackage[cp1251]{inputenc}
\usepackage[TS1,T2A]{fontenc}
\usepackage[russian]{babel}
\usepackage{url}
\usepackage{amssymb,latexsym,makeidx,amsmath,amsfonts,amsthm}
\usepackage{hyperref}
\usepackage{mathdots}

\usepackage{color}

\newcommand\ISH[1]{}
\newcommand\AK[1]{}
\newcommand\store[1]{}

\newtheorem{theorem}{Теорема}
\newtheorem{nontheorem}{Теорема}
\newtheorem{lemma}[nontheorem]{Лемма}
\newtheorem{proposition}[nontheorem]{Предложение}

\newtheorem{corollary}[nontheorem]{Следствие}
\newtheorem*{corollary*}{Следствие}

\newtheorem{question}{Вопрос}

\newtheorem*{conjecture*}{Гипотеза}

\theoremstyle{definition}
\newtheorem{definition}[nontheorem]{Определение}

\newcommand\cone[2]{{#1}\langle #2 \rangle}

\newcommand\abs[1]{|#1|}
\newcommand\un[1]{\underline{#1}}
\def\emp{\varnothing}
\newcommand\str[1]{{\bf S}(#1)}

\def\subgen{\sqsubseteq}

\newcommand\logicts[1]{{\textsc{#1}}}
\newcommand{\LK}[1]{\logicts{K#1}}
\newcommand{\LS}[1]{\logicts{S#1}}

\newcommand\framets[1]{{#1}}
\newcommand\modelsts[1]{\framets{#1}}

\newcommand\hide[1]{}
\newcommand\temp[1]{{#1}^\pm}

\renewcommand{\ge}{\geqslant}
\renewcommand{\le}{\leqslant}
\renewcommand{\geq}{\geqslant}
\renewcommand{\leq}{\leqslant}

\def\vf{\varphi}
\def\mM{\modelsts{M}}

\def\AA{\forall}
\def\EE{\exists}

\def\iff{\Leftrightarrow}

\def\dts{\dots}
\def\Di{\lozenge}
\def\imp{\rightarrow}
\def\Imp{\Rightarrow}
\def\Lvar{\logicts{L}}
\def\mo{\vDash}

\newcommand{\set}[1]{\left\{#1\right\}}

\newcommand{\ZZ}{\mathbb{Z}}

\def\ov{\overline}

\def\restr{{\upharpoonright}}

\newcommand{\clF}{\mathcal{F}}
\newcommand{\clG}{\mathcal{G}}
\newcommand{\clA}{\mathcal{A}}
\newcommand{\clB}{\mathcal{B}}
\newcommand{\clC}{\mathcal{C}}
\newcommand{\Log}{\mathop{Log}}

\def\Dim{\Di^{\leq m}}
\def\Boxm{\Box^{\leq m}}

\def\h{\mathop{h}}
\def\l{\mathop{l}}

\def\subgen{\sqsubseteq}

\def\nin{\notin}
\def\con{\wedge}

\def\mo{\models}
\def\trm{\mathop{trans}}

\date{}
\title{О разбиениях шкал Крипке конечной высоты}
\author{
  Кудинов А. В. \\
    ИППИ РАН
  \and
  Шапировский И. Б. \\
    ИППИ РАН
}

\begin{document}

\maketitle

\hide{УДК 510.643}


\begin{abstract}
В работе доказана финитная аппроксимируемость и разрешимость одного семейства модальных логик.
Бинарное отношение $R$ назовём предтранзитивным, если $R^*=\cup_{i\leq m} R^i$ для некоторого $m\geq 0$,
где $R^*$ --- транзитивное рефлексивное замыкание $R$.
Под высотой шкалы $(W,R)$  будем понимать высоту предпорядка $(W,R^*)$.
Описаны специальные разбиения (фильтрации) предтранзитивных шкал конечной высоты, из чего следует финитная аппроксимируемость и разрешимость их модальных логик.
\hide{
The paper proves finite model property and decidability for a family of modal logics.
A binary relation $R$ is called pretransitive, if $R^*=\cup_{i\leq m} R^i$ for some $m\geq 0$,
where $R^*$ is the transitive reflexive closure of $R$.
By the height of $(W,R)$ we mean the height of the preorder $(W,R^*)$.
Special partitionings (filtrations) are described for pretransitive frames of finite height,
that\ISH{which?} implies finite model property and decidability of modal logics of these frames.}
\end{abstract}



\hide{  
Отношение $R$ назовём предтранзитивным, если $ $ для некоторого $m$.
В рассматриваются модальные логики предтранзитивных шкал конечной высоты.
Для ряда таких логик доказана финитная аппроксимируемость и алгоритмическая разрешимость. В частности, показано, что
$(m,n)$-шкалы конечной высоты ... 

Хорошо известно, что логики предпорядков конечной высоты являются финитно аппроксимируемыми и разрешимыми. Мы рассматриваем
так называемые предтранзитивные отношения, то есть те, в которых транзитиное замыкание которых содержится

Финитная аппроксимируемость логик $(m,n)$-шкал является старой открытой проблемой.
Нам удалось построить подходящие разбиения для случая предтранзитивных $(m,n)$-шкал конечной высоты. 

}

\section{Введение и основные результаты}
Язык пропозициональной модальной логики --- это язык классической логики высказываний с дополнительными связками.
Несмотря на простоту, он оказывается эффективным средством описания свойств отношений.
Возникающие при этом теории зачастую обладают лучшими свойствами
(такими, например, как алгоритмическая разрешимость или невысокая сложность),
чем соответствующие теории первого порядка.
Это обусловило широкое применение модальных логик в информатике (см., например, \cite{BaaderDL2003,Vardi06automata}), а также
в других областях математической логики: при изучении фрагментов предикатных логик, в теории доказательств, теории множеств, алгебраической логике (см., например, \cite{ModalMogel,BoolosBook,SmullFitt,venema-cyl1995}).
\store{
Помимо задач, пришедших из приложений, траляля, есть ещё свои собственные траляля. Одна такая связана с финитной аппроксимируемостью и разрешимостью ....
}

\paragraph{Алгебраические и реляционные модели.}
Первые модальные исчисления возникли в 1930-х гг. как формализация понятия возможной истинности \cite{Lewis-Langford}.
Довольно быстро выяснилось, что модальные системы имеют естественную алгебраическую интерпретацию, при которой логика оказывается эквациональной теорией некоторой
алгебры:  выводимость формулы в исчислении
соответствует истинности тождества в алгебре.
Центральными объектами оказались {\em модальные алгебры} --- булевы алгебры с дополнительными аддитивными операциями,
и соответствующие теории ---
{\em модальные логики} (более точно ---  {\em нормальные пропозициональные модальные логики}).
Например, модальной алгеброй является алгебра замыкания топологического пространства ---
булева алгебра его подмножеств с дополнительной одноместной операцией, отображающей множество в его замыкание.\ISH{Нужен этот пример?}

Изучение модальных алгебр и их теорий началось с 1940-х годов в работах Тарского, МакКинси, Йонссона и других:
в \cite{McKinsey-Tarski48} две из описанных в \cite{Lewis-Langford} дедуктивных системы  ($\LS{4}$ и $\LS{5}$)
были рассмотрены как теории алгебр замыкания и монадических алгебр;
\store{
Были описаны первые теории: теория класса алгебр замыкания и теория класса монадических алгебр \cite{McKinsey-Tarski48}
\ISH{ТАВТОЛОГИЯ}
(ими оказались два из предложенных в \cite{Lewis-Langford} исчисления,
которые традиционно обозначаются $\LS{4}$ и $\LS{5}$);
}
в \cite{McKinsey41} были доказаны первые результаты о разрешимости модальных логик.
В работе \cite{JonssonTarski} Йонссоном и Тарским для модальных алгебр
была доказана  теорема о представлении --- обобщение теоремы Стоуна о представлении булевых алгебр.

\ISH{Это --- после Крипке?}
В отличие от булевых алгебр,
имеющих в нетривиальном случае общую эквациональную теорию (классическую логику высказываний),
модальные алгебры приводят к более разнообразным теориям:
уже в случае с  единственной \ISH{одноместной} модальной операцией возникает континуум различных модальных логик \cite{JankovContin68}.
\ISH{Проверить ссылку.}
Из этого, в частности, следует, что среди модальных логик имеются алгоритмически неразрешимые.

\bigskip

Принципиально важной для дальнейшего развития модальной логики оказалась {\em реляционная семантика}, или {\em семантика Крипке}, предложенная \ISH{обнаруженная, описанная} в
конце 1950-х гг.
\store{Основные приложения модальной логики связаны с {\em реляционной семантикой}, или {\em семантикой Крипке}, возникшей в конце 1950-х гг.}
Непустое множество с отношениями на нём называется {\em шкалой (Крипке)}.
\store{(в контексте модальной логики множества с отношениями называют шкалами Крипке).}
Нас будет в основном интересовать случай, когда в шкале имеется единственное бинарное отношение.
Такой шкале $(W,R)$ соответствует модальная алгебра $A(W,R)$ --- булева алгебра всех подмножеств $W$ с дополнительной одноместной операцией $\Di_R$:
для $V\subseteq W$  $\Di_R(V)$ --- прообраз $V$ относительно $R$:
 $$\Di_R(V)=\{w \mid \EE v\in V ~wRv\}.$$
Модальная логика шкалы или класса шкал --- это \ISH{, фактически,}эквациональная теория соответствующих \ISH{модальных} алгебр.
\store{
{\em Модальная логика класса шкал} --- это эквациональная  теория соответствующего класса алгебр.
{\em Модальная логика шкалы} $(W,R)$ --- это совокупность всех тождеств алгебры $A(W,R)$, то есть её
 эквациональная теория.
Аналогично, {\em модальная логика класса шкал} --- это эквациональная  теория соответствующего класса алгебр.
}

Интерпретация модальных формул на множествах с отношениями
привела к возникновению разнообразных приложений модальной логики.
\store{обусловила \ISH{самое разнообразное; многочисленное} разнообразное применение модальных логик в приложениях.
\ISH{Отвратительно пафосная фраза}}
\store{Вкратце упомянем некоторые из них.}
Так, любая реляционная структура является моделью языка первого порядка.
Это позволяет смотреть на модальные логики как на фрагменты первопорядковых теорий;  во многих случаях
такие фрагменты оказываются разрешимыми \cite{ModalMogel}.
\store{Возникли применения модальных логик в информатике: }
Шкалу $F$ (с оценкой в $A(F)$) можно рассматривать как систему переходов некоторой вычислительной системы;
в этом случае модальные формулы описывают свойства \ISH{(``поведение'')} вычислительного процесса \cite{Vardi06automata}.
\ISH{Guarded?}
Ещё одно важное применение модальные логики находят в теории доказательств.
Уже в 1930-х годах  Гёделем  было предложено использовать модальные операторы для описания доказуемости и непротиворечивости
в формальной арифметике \cite{Goedel33}. Полная аксиоматика арифметической доказуемости в модальном языке была построена (существенно позже)
Соловеем \cite{Solovay76}; в реляционной семантике эта логика задаётся
конечными  \ISH{строгими} частичными порядками, из чего следует её разрешимость.

\paragraph*{Финитная аппроксимируемость и разрешимость.}
Вопрос об алгоритмической разрешимости --- один из ключевых при изучении модальных логик.
Во многих случаях разрешимость следует из финитной аппроксимируемости (теорема Харропа).
\store{Важным свойством модальных логик является финитная аппроксимируемость.

В частности, во многих случаях из неё немедленно следует
алгоритмическая разрешимость.}
{\em Финитно аппроксимируемая логика} --- это логика некоторого класса конечных шкал
(это соответствует финитной аппроксимируемости её свободной алгебры).
Финитная аппроксимируемость модальных логик систематически изучалась с середины 1960-х годов \cite{LemmonScott1966,Segerberg1968,Gabbay:1972:JPL}.
Известно, что многие логики обладают этим свойством.
Логики без финитной аппроксимируемости тоже существуют (более того --- их континуум, см., например, \cite[Theorem 6.1]{Ch:Za:ML:1997}), но в случае с единственным
модальным оператором они
скорее экзотика и обыкновенно строятся искусственно.  В то же время, имеется ряд естественным образом задаваемых логик, вопрос о финитной
аппроксимируемости (и разрешимости) которых открыт.  \store{Чуть ниже мы приведём примеры таких логик.}

По-видимому, самым известным таким  вопросом является финитная аппроксимируемость и разрешимость логик {\em $(m,n)$-шкал}, то есть шкал, в которых отношение удовлетворяет условию  $R^n\subseteq R^m$.
Например, таким условием является транзитивность $R\circ R\subseteq R$. В этом случае ($m=1$, $n=2$),
как и во всех случаях, когда один из параметров меньше $2$,
или же в тривиальном случае
$m=n$, финитная аппроксимируемость и разрешимость известны с начала 1970-х годов \cite{Gabbay:1972:JPL}. Финитная
аппроксимируемость и разрешимость в остальных случаях являются старой открытой проблемой, см. \cite[Problem 11.2]{Ch:Za:ML:1997},
\cite[Problem 6]{ModalDecWZ} (в последней формулировке эта
задача названа одной из наиболее интересных задач модальной логики --- дословно,
“\textit{one of the most intriguing open problems in Modal Logic}”). Ответ неизвестен ни для каких $m,n>1$, $m\neq n$.

Класс всех $(m,n)$-шкал обозначим $\clF(m,n)$.
\begin{question}
При каких $m,n$ логика класса $\clF(m,n)$ финитно аппроксимируема? Разрешима?
\end{question}

Ещё один пример логик, финитная аппроксимируемость и разрешимость которых неизвестна, дают так называемые {\em предтранзитивные шкалы}, то есть шкалы, в которых для некоторого $m\geq 0$ имеется свойство {\em $m$-транзитивности}:
$R^*=\cup_{i\leq m} R^i.$
Ответ положителен при $m=0$ и $m=1$, в остальных случаях вопрос открыт.

$\clG(m)$ обозначает класс всех $m$-транзитивных шкал.
\begin{question}
При каких $m$ логика класса $\clG(m)$ финитно аппроксимируема? Разрешима?
\end{question}

\bigskip
Финитную аппроксимируемость в ряде случаев можно установить путём построения специальных разбиений шкал.

Пусть $\clB$ --- некоторое разбиение множества $W$, $R$ --- бинарное отношение на $W$. На элементах разбиения $\clB$ определим отношение
$R_\clB$, положив для $U,V\in \clB$   $$U ~R_\clB~ V \iff \exists u\in U ~\exists v \in V ~u R v.$$
Известен следующий факт:
логика класса шкал $\clF$ оказывается финитно аппроксимируемой, если для любого конечного разбиения $\clA$
любой шкалы $(W,R)\in \clF$ найдется его конечное измельчение $\clB$ такое, что $(\clB,R_\clB)\in \clF$.
В этом случае говорят, что $\clF$ {\em допускает минимальные фильтрации}.

Приведём ещё одно достаточное условие финитной аппроксимируемости.
Разбиение $\clB$ множества $W$ назовём {\em правильным}, если для любых $U,V\in \clB$ выполнено следующее условие:
$$\EE u\in U ~\EE v\in V ~uRv ~~\Imp ~~ \AA u\in U ~\EE v\in V ~uRv.$$
Шкала называется {\em правильно измельчаемой}, если у любого её конечного разбиения существует правильное конечное измельчение.
Логика всякого класса правильно измельчаемых шкал является финитно аппроксимируемой.

В случае, когда $\clF$ --- класс всех шкал некоторой логики, второе условие является более сильным: если все шкалы из $\clF$ измельчаемы, то этот класс допускает минимальные фильтрации.

Заметим, что приведённые выше условия не требуют изучения аксиоматики или каких-то других синтаксических свойств рассматриваемой логики,
и позволяют получать результаты о финитной аппроксимируемости чисто семантически.
\ISH{Надо это вообще? Вроде это общее место  в общем алгебраическом контексте.}

\paragraph{Основные результаты.}
В настоящей работе строятся разбиения предтранзитивных шкал {\em конечной высоты}.
\store{Нас будут интересовать разбиения {\em шкал конечной высоты}.}
Частичный порядок имеет {\em конечную высоту} $h$, если $h$ --- наибольшая из мощностей цепей в этом порядке.
С каждой шкалой $(W,R)$ можно естественным образом связать частичный порядок, элементами которого будут классы эквивалентности
по $\sim_R$:
\[
w \sim_R v \iff w R^* v \mbox{ и } v R^* w,
\]
где  $R^*$ --- транзитивное рефлексивное замыкание $R$; эти классы упорядочены отношением $\leq_R$:
$[x] \leq_R [y] \iff xR^*y$.
{\em Высота шкалы} --- это высота соответствующего ей частичного порядка.

В теореме \ref{thm:finite_depth_part} установлено, что всякая шкала конечной высоты, в которой все
классы эквивалентности $\sim_R$ конечны и ограничены в совокупности, является правильно измельчаемой.
Эта ключевая теорема позволяет доказывать устойчивость относительно минимальных фильтраций  различных классов шкал.

Заметим, что при $n>m$ $(m,n)$-шкала является $(n-1)$-транзитивной.
Нам удалось построить нужные разбиения для всех шкал из классов $\clF(m,n)$ и $\clG(m)$, $n > m\geq 1$, высота которых конечна.
Пусть $\clF(m,n,h)$ --- все шкалы высоты не более $h$ из класса $\clF(m,n)$, $\clF_*(m,n)$ --- все $(m,n)$-шкалы конечной высоты.
Аналогично определим классы $m$-транзитивных шкал конечной высоты $\clG(m,h)$ и $\clG_*(m)$.

Теорема \ref{thm:fmp-two-classes} устанавливает  финитную аппроксимируемость логик классов $\clF(m,n,h)$,  $\clG(m,h)$,
$\clF_*(m,n)$ и $\clG_*(m)$ при $n > m\geq 1$, $h\geq 1$.
\begin{corollary*} Для $n > m\geq 1$,
логика класса $\clF(m,n)$ финитно аппроксимируема тогда и только тогда, когда она совпадает с логикой класса $\clF_*(m,n)$.
Аналогичный критерий имеет место для логик классов $\clG(m)$.
\end{corollary*}


\bigskip
Дальнейшие результаты связаны с исследованием логик предтранзитивных шкал  как дедуктивных исчислений.

Теорема \ref{thm:Glivenko} является аналогом теоремы Гливенко для предтранзитивных логик.
Напомним, что в семантике Крипке интуиционистская логика --- это логика  частичных порядков, а классическая --- логика порядков высоты 1;
теорема Гливенко утверждает, что выводимость формулы $\vf$ в классической логике высказываний равносильна выводимости $\neg\neg\vf$ в интуиционистской логике.
Теорема \ref{thm:Glivenko} описывает аналогичную сводимость для предтранзитивных логик.

Теорема \ref{thm:compl} описывает модальные аксиоматики для классов предтранзитивных шкал конечной высоты. В частности, из неё
вытекает разрешимость логик $\clF(m,n,h)$ и $\clG(m,h)$ для всех $n > m\geq 1$, $h\geq 1$.



\medskip
Работа организована следующим образом. В разделе \ref{sec:prel}
содержатся некоторые предварительные сведения.
В разделе \ref{sect:part} строятся правильные разбиения шкал конечной высоты.
В разделе \ref{sect:twofam} строятся разбиения шкал конечной высоты из классов $\clF(m,n)$ и  $\clG(m)$.
В разделе \ref{sect:ax} рассматриваются вопросы модальной аксиоматизации.
Обсуждение результатов и некоторые их следствия приведены в разделе \ref{sect:cor}.

\hide{
=====
Это свойство выражается некоторой модальной формулой. В логиках, содержащих формулу предтранзитивности, выражается свойство конечной высоты шкалы.
Финитная аппроксимируемость конечно аксиоматизируемой логики даёт её разрешимость. Полные

В теореме \ref{} установлена разрешимость логик $\clF(m,n,h)$,  $\clG(m,h)$, для всех $n > m\geq 1$, $h\geq 1$.

Для логик $m$-транзитивных логик мы докажем вариант теоремы Гливенко: аналогично тому, как классическая логика высказываний сводится к интуиционистской,
$m$-транзитивная логика сводится к своему расширению формулы высоты 1 (напомним, что в семантике Крипке для интуициониских логик
интуициониская логика --- это логика всех частичных порядков, а классическая --- логика тривиальных порядков --- высоты 1). Из этой сводимости мы
и получим необходимое условие финитной аппроксимируемости и разрешимости (сформулировать?).

}

\section{Предварительные сведения}\label{sec:prel}

\subsection{Язык и семантика}
Множество \emph{модальных формул} строится из счетного множества
переменных $PV=\set{p_1,\, p_2,\, \ldots
}$ и констант $\bot$ ({\em ``ложь''}), $\top$ ({\em ``истина''}) с помощью  связок $\land$, $\lor$, $\lnot$, $\to$, $\leftrightarrow$ и
одноместной связки $\Di$ ({\em``ромб''}).

Дедуктивное понятие модальной логики мы введём позже, в разделе \ref{sect:ax}, где
рассматриваются вопросы аксиоматизации. Для нужд этого и следующих двух разделов нам будет достаточно понятия {\em логики шкал Крипке}.

\emph{Шкалой (Крипке)} называется пара $(W,R)$, где $W$ "---
непустое множество и $R \subseteq W\times W$. \emph{Оценкой} на
шкале называется отображение $\theta: PV\to \mathcal{P}(W),$ где
$\mathcal{P}(W)$ обозначает множество всех подмножеств $W$.
\emph{Моделью}  $M$ над шкалой $(W,R)$ называется тройка
$(W,R,\theta)$, где $\theta$ "--- оценка на $(W,R)$.
Истинность модальной формулы в точке модели определяется индукцией по длине формулы и
обозначается $M,w\mo \vf$:  для переменной $p$ полагаем $M,w\mo p  ~\iff ~ w\in \theta(p)$;
булевы связки интерпретируются стандартно, например, $M,w\mo \neg\vf ~\iff  ~\mM,w\not\mo \vf$;
$$\mM,w\mo \Di \vf  ~\iff ~  \EE v (wRv \textrm{ и } M,v\mo \vf).$$
Формула $\vf$ \emph{истинна в модели}
$M$, если она истинна в любой точке модели $M$; $\vf$
\emph{общезначима в шкале} $F$, если она истинна в любой модели над
$F$; $\vf$ \emph{общезначима в классе шкал} $\clF$, если $\vf$
общезначима в каждой шкале из $\clF$; в обозначениях: $M\mo\vf$,
$F\mo\vf$, $\clF\mo\vf$ соответственно.
Для множества формул $\Phi$, $F\mo \Phi$ означает $F \mo\vf$ для всех $\vf\in \Phi$; в этом случае будем говорить, что  $F$ --- {\em $\Phi$-шкала}.
Аналогично понимаются записи $M,x\mo\Phi$, $M\mo\Phi$, $\clF\mo\Phi$.

Множество всех формул, общезначимых в классе $\clF$, обозначается $\Log \clF$ и называется {\em логикой\footnote{Во вводном разделе логика шкал определялась как эквациональная теория.
Строго говоря, имеется формальное различие между языком модальных формул и языком тождеств модальной сигнатуры, однако оно несущественно:
общезначимость модальной формулы $\vf$ в шкале $(W,R)$ равносильна тому, что в алгебре $A(W,R)$ при любой оценке значение $\vf$ совпадает с $W$ --- единицей алгебры;
то, что в алгебре $A(W,R)$ имеет место тождество $\vf=\psi$, означает общезначимость в шкале $(W,R)$ модальной формулы $\vf\leftrightarrow\psi$.
}
 $\clF$}.

\hide{
Во вводном разделе мы дали более лаконичное, хотя и менее наглядное определение логики шкал. Легко убедиться, что эти определения эквивалентны: общезначимость модальной формулы $\vf$ в шкале $(W,R)$ равносильна тому, что в алгебре $A(W,R)$ при любой оценке значение $\vf$ совпадает с $W$ --- единицей алгебры.
(Имеется формальное различие между языком модальных формул и языком тождеств модальной сигнатуры, однако оно несущественно в силу следующего тривиального соображения: то, что в алгебре $A(W,R)$ имеет место тождество $\vf=\psi$, означает общезначимость в шкале $(W,R)$ модальной формулы $\vf\leftrightarrow\psi$).
\ISH{Нужно ли это замечание?}}


\subsection{Разбиения и фильтрации}

Логики классов конечных шкал называются {\em финитно аппроксимируемыми}.

Формула называется \emph{выполнимой в шкале} $F$ ({\em в классе шкал} $\clF$), если
она истинна в некоторой точке некоторой модели над $F$ (над некоторой
шкалой из $\clF$).

Финитная аппроксимируемость $\Lvar=\Log \clF$  эквивалентна тому, что
всякая выполнимая в $\clF$ формула выполнима в некоторой конечной $\Lvar$-шкале.
Одним из способов нахождения такой шкалы является метод фильтраций.



\begin{definition}\label{def:epifiltr}
Рассмотрим шкалу $F = (W,R)$ и отношение эквивалентности $\sim$ на $W$. {\em Минимальной фильтрацией по $\sim$} (или \emph{$\sim$"=фильтрацией) {\em шкалы} $F$} называется шкала
$F/{\sim} = (W/{\sim}, R/{\sim})$, где для $U,V\in W/{\sim}$ полагаем
$$ U~R/{\sim}~V ~\iff ~ \exists u\in U ~\exists v \in V ~u R v.$$
\end{definition}

\begin{definition}
Пусть $M$ --- модель, $\vf$ --- формула.
Определим эквивалентность $\sim_\vf$, {\em индуцированную формулой} $\vf$ на точках $M$: положим
$u\sim_\vf v$ в случае, когда
каждая подформула $\vf$ одновременно истинна или одновременно ложна в $u$ и $v$.

Скажем, что эквивалентность $\sim$ {\em согласована с формулой $\vf$ в модели $M$}, если
$\sim ~\subseteq ~\sim_\vf$.
\end{definition}

Число подформул $\vf$ обозначим $\l(\vf)$. Очевидно, что в любой модели число классов $\sim_\vf$ не превышает $2^{\l(\vf)}$.




\begin{proposition}[Лемма о минимальной фильтрации, см., например, \cite{shehtman2004filtration}] \label{prop:epifiltr}
Если $\vf$ истинна в одной из точек модели $M$ над шкалой $F$ и эквивалентность $\sim$ согласована с $\vf$ в $M$, то $\vf$ выполнима в $F/{\sim}$.
\end{proposition}

Приведённое выше определение минимальной фильтрации --- частный случай более общей конструкции ({\em фильтрации модели Крипке}). Фильтрации возникли в конце 1960-х в работах
\cite{LemmonScott1966,Segerberg1968} и в дальнейшем стали одним из основных инструментов доказательства финитной аппроксимируемости модальных логик
\cite{Ch:Za:ML:1997}.
В случае минимальных фильтраций этот метод сводится к построению подходящих разбиений шкал.
Как обычно, {\em (конечным) разбиением множества} $W$ называется (конечное) семейство попарно непересекающихся непустых множеств, объединение которых совпадает с $W$.
Если $\clA$ и $\clB$ --- разбиения $W$, и $\AA B\in\clB~\EE A\in\clA~B\subseteq A$, то $\clB$ называется {\em измельчением} $\clA$.
Через $\sim_\clA$ будем обозначать отношение эквивалентности, множество классов которого совпадает с $\clA$: $\clA=W/{\sim_\clA}$.
Таким образом, точками шкалы $F/{\sim_\clA}$ являются элементы разбиения $\clA$. Вместо $F/{\sim_\clA}$ и $R/{\sim_\clA}$  будем
писать $F_\clA$ и $R_\clA$.

\begin{definition} Класс шкал $\clF$ {\em допускает минимальные фильтрации}, если для всякой шкалы $F\in\clF$, для всякого конечного разбиения носителя $F$ найдется конечное измельчение этого разбиения $\clB$ такое, что
$F_\clB\in \clF$. Если более того существует $f:\mathbb{N}\to \mathbb{N}$ такая, что
 для всякой шкалы $F\in\clF$, для всякого конечного разбиения $\clA$ носителя $F$ найдется измельчение $\clB$ такое, что
$F_\clB\in \clF$ и
$$|\clB|\leq f(|\clA|),$$ то $\clF$ {\em допускает $f$-ограниченные минимальные фильтрации}.
\end{definition}
\ISH{Раньше было $f$-минимальные --- это очень плохо. Сейчас длинно. Есть альтернативы?}

\begin{proposition}\label{prop:admitsfiltr-fmp}
Если $\clF$  допускает минимальные фильтрации, то $\Log \clF$ финитно аппроксимируема.
Если $\clF$ допускает $f$-ограниченные минимальные фильтрации, то
всякая выполнимая в $\clF$ формула $\vf$ выполнима в шкале из $\clF$, размер которой не превышает
$f(2^{\l(\vf)})$.
\end{proposition}
\begin{proof}
Пусть $\vf$ выполнима в классе $\clF$. Тогда $\vf$ истинна в одной из точек некоторой модели $M$ над некоторой шкалой $F\in\clF$.
Рассмотрим на $M$ эквивалентность $\sim_\vf$.
Пусть $\clA$ --- классы эквивалентности $\sim_\vf$. По условию леммы, у разбиения $\clA$ существует конечное измельчение $\clB$ такое, что $F_\clB\in \clF$.
То, что $\clB$ --- измельчение $\clA$, равносильно тому, что $\sim_\clB~\subseteq ~\sim_\clA$.
По лемме о фильтрации, $\vf$ выполнима в $F_\clB$.

Если при этом $|\clB|\leq f(|\clA|)$, то, поскольку $|\clA|\leq 2^{\l(\vf)}$,
размер $\clB$ (то есть размер построенной шкалы) не превышает $f(2^{\l(\vf)})$.
\end{proof}

\subsection{Правильные разбиения шкал}
\begin{definition}
Пусть $F=(W,R)$ --- шкала. Разбиение $\clA$ множества $W$ назовём {\em правильным}, если для любых $U,V\in \clA$ выполнено следующее условие:
$$\EE u\in U ~\EE v\in V ~uRv ~~\Imp ~~ \AA u\in U ~\EE v\in V ~uRv.$$
Отношение эквивалентности на $W$ {\em правильно разбивает} $F$, если множество его классов эквивалентности является правильным разбиением.
\ISH{Левое условие --- через $R_\clA$? То же --- во введении?}

Шкала называется {\em правильно измельчаемой}, если у любого её конечного разбиения существует правильное конечное измельчение.
Шкала называется {\em $f$-правильно измельчаемой}, где $f:\mathbb{N}\to \mathbb{N}$,
если у любого её конечного разбиения $\clA$ существует правильное конечное измельчение $\clB$ такое, что $|\clB|\leq f(|\clA|)$.
\end{definition}

Известен следующий факт (впервые, по-видимому, он был отмечен в \cite{Franz-Bull}).
\ISH{Я так и не добыл эту заметку.}
\begin{proposition}\label{prop:franz-pmorph}
Если $\clA$ --- правильное разбиение $F$,
то $\Log \{F\} \subseteq \Log \{F_\clA\} $.
\end{proposition}
\begin{proof}
Несложно проверить, что если $\clA$ --- правильное разбиение, то
алгебра $A(F_\clA)$ вкладывается в алгебру $A(F)$.
\end{proof}
Из этого факта следует
\begin{proposition}
Если $\clF$ --- класс правильно измельчаемых шкал, то логика $\Log \clF$ финитно аппроксимируема.
Если $\clF$ --- класс $f$-правильно измельчаемых шкал, то
всякая выполнимая в $\clF$ формула $\vf$ выполнима в $\Log \clF$-шкале, размер которой не превышает
$f(2^{\l(\vf)})$.
\end{proposition}
\begin{proof}
Аналогично доказательству предложения \ref{prop:admitsfiltr-fmp}.
\end{proof}

\begin{proposition}
Пусть $\Lvar=\Log \clF$, и $\clF$ --- все $\Lvar$-шкалы. Если класс $\clF$ состоит из правильно измельчаемых шкал, то он допускает минимальные фильтрации.
\end{proposition}
\begin{proof}
Следует из предложения \ref{prop:franz-pmorph}.
\end{proof}

\subsection{Шкалы конечной высоты}
Пусть $R$ "--- бинарное отношение на $W$.
Транзитивное рефлексивное замыкание $R$ обозначается $R^*$.
Напомним, что
\begin{align*}
 R^0 &= Id(W) = \{(x,x)\mid x\in W\},\\
 R^{i+1} &= R^i \circ R,\\
 R^* &= \bigcup_{i \ge 0} R^i.
\end{align*}

Рассмотрим шкалу $F=(W,R)$.  Определим следующее отношение на $W$:
\[
x \sim_R y \iff x R^* y \mbox{ и } x R^* y.
\]
\ISH{Раньше элементами шкал были $u,v$.}
Это отношение является отношением эквивалентности: оно симметрично
по определению, а транзитивность и рефлексивность наследуются от
отношения $R^*$. Классы эквивалентности по $\sim_R$ называются
\emph{сгустками} шкалы $F$, а шкала $skF  = (W/ {\sim_R}, \leq_R)$, где
$$[x]\leq_R [y] \iff xR^*y,$$ называется \emph{остовом шкалы} $F$
($[z]$ тут
обозначает  $\sim_R$"=класс эквивалентности $z$). Такое определение
отношения $\leq_R$ корректно, так как если $x'\sim_R  x$ и $y'\sim_R y$, то
$$xR^*y \iff x' R^*y'.$$
Легко проверяется, что $\leq_R$ "--- частичный порядок. Кроме того,
если $F$ "--- частичный порядок, то $F$ и $skF$ изоморфны. Положим
также $[x]<_R [y]$ $\iff$ $[x]\leq_R [y]$ и $[x]\neq [y]$ (последнее
в данном случае означает, что $(y,x)\not\in R^*$).

Частичный порядок имеет \emph{высоту} $h$, если в нём существует
цепь мощности $h$ и нет цепей большей мощности. Под \emph{высотой}
$\h(F)$ произвольной шкалы $F$ будем понимать высоту её остова.

Через $\h(F,x)$ обозначим {\em глубину точки $x$ в шкале $F$}, то есть  высоту сужения $F\restr\{y \mid x  R^* y\}$.  Напомним, что
{\em cужением $R$ на $V\subseteq W$} называется отношение $R\restr V=R\cap (V\times V)$, \emph{сужением шкалы}
$F=(W,R)$ на непустое $V\subseteq W$ называется шкала $F\restr V=(V,R\restr
V)$.

Совокупность точек глубины $i$ в $F$ называется {\em $i$-м слоем} $F$.
Таким образом, шкала высоты $h$ разбивается на $h$ слоёв.  Сгустки, соответствующие точкам из одного слоя,  составляют в остове $F$  антицепь.

Заметим, что $\h(F)=1$ равносильно тому, что $R^*$ --- отношение эквивалентности. \ISH{Надо это тут?}

\bigskip
Логики предпорядков конечной высоты хорошо изучены \cite{Seg_Essay,Max1975}. В частности, известна их финитная аппроксимируемость.
Нас будет интересовать более общий случай --- когда имеется лишь свойство предтранзитивности:
%
шкала называется {\em предтранзитивной}, если она является {\em $m$-транзитивной} для некоторого $m\geq 0$, то есть обладает свойством $$R^*= R^{\leq m},$$
где $$R^{\leq m}=  \bigcup\limits_{0\leq i \leq m} R^i.$$
Несложно проверить, что $m$-транзитивность равносильна свойству $$R^{m+1}\subseteq R^{\leq m}.$$

Предтранзитивность следует из многих естественных свойств.
Очевидно, что любая конечная шкала является предтранзитивной. Более того, если $F$ --- шкала конечной высоты $h$ и мощности её сгустков ограничены
в совокупности некоторым конечным $N$, то эта шкала является предтранзитивной (более точно --- $hN$-транзитивной).

Заметим ещё, что если $R^n\subseteq R^m$  при $n>m$, то шкала является $(n-1)$-транзитивной. Таким образом, при $n>m$ $\clF(m,n)\subseteq \clG(n-1)$.

\section{Правильные разбиения шкал конечной высоты}\label{sect:part}
Пусть $\exp^i_2(x)$ обозначает башню экспонент высоты $i$:
\[
 \Bigl.
\begin{array}{c}
\exp^i_2(x) =
{
{2^{\displaystyle{2^{\scriptscriptstyle{\iddots^{\displaystyle{2^x}}}}}}}
}
\end{array}
\biggr\} i\ \hbox{раз}
\]
\begin{theorem}\label{thm:finite_depth_part}
Всякая шкала конечной высоты, в которой мощности сгустков ограничены
в совокупности некоторым конечным $N$,
является правильно измельчаемой.
Более того, если $h$ --- высота такой шкалы, то
она $f$-правильно измельчаема, где
$$f(x)=\exp^h_2\bigl( (N+h+1)(\log_2 x+N)\bigr).$$
\end{theorem}
\begin{proof}
Пусть $F=(W,R)$ --- шкала, удовлетворяющая условиям теоремы, $\clA$ --- некоторое конечное разбиение $W$.
Определим по индукции последовательность эквивалентностей $$\sim_0~\subseteq ~\dots ~\subseteq ~\sim_h,$$ правильно разбивающих $W$, и соответствующих правильных разбиений
$\clB_i:=W/{\sim_i}$. Каждое из этих разбиений будет измельчением $\clA$, а последнее разбиение $\clB_h$ окажется конечным.

Будем писать $\h(x)$ вместо $\h(F,x)$. Положим $$X_i=\{x \mid \h(x)>i\},~ Y_i=\{x \mid \h(x)=i\},
~Z_i=\{x \mid \h(x)<i\},$$
то есть $Y_i$ --- $i$-й слой шкалы $F$, $Z_i$ --- точки глубины меньшей $i$, $X_i$ --- точки глубины большей $i$.

Положим $$\sim_0~:=~Id(W),$$ т.е. $\clB_0$ --- разбиение $W$ на синглетоны. Это разбиение очевидно является правильным измельчением $\clA$.

Пусть $1\leq i\leq h$. Определим $\sim_i$. Эта эквивалентность расширит $\sim_{i-1}$ некоторыми парами точек из
 $i$-го слоя, и при этом разобьёт его на конечное число классов. Для этого нам потребуются некоторые вспомогательные конструкции. Во-первых, выделим ``верхнюю'' часть разбиения $\clB_{i-1}$:
$$\clB'_{i-1}=\{B\in\clB_{i-1}\mid B\subseteq Z_{i}\}.$$
Теперь зафиксируем сигнатуру $\Omega_i$, состоящую из одного бинарного предикатного символа,
одноместных предикатных символов $\un{P}_B$ для всех $B \in \clB'_{i-1}$, одноместных предикатных символов $\un{T}_A$ для всех $A \in \clA$ и символа константы.
Для каждого элемента $u\in Y_i$ определим $\str{u}$ --- структуру сигнаруты $\Omega_i$. Её носителем будет сгусток $C$ шкалы $F$, содержащий точку $u$.
Бинарным отношением на $C$ будет сужение $R\restr C$.
Для каждого $B\in \clB'_{i-1}$ определим подмножество $C$, интерпретирующее $\un{P}_B$:
$$P^u_B:=\{w\in C \mid \EE v\in B ~ wRv \}$$
(то есть $P^u_B$ --- пересечение $C$ с прообразом $B$ относительно $R$).
Для каждого $A\in \clA$ в качестве множества, интерпретирующего $\un{T}_A$ возьмём
$T^u_A:= C \cap A$.
В качестве константы выберем $u$.
Таким образом,
$$\str{u}:=(C, R\restr C, (P^u_B)_{B\in \clB'_{i-1}}, (T^u_A)_{A\in \clA}, u).$$
\ISH{А правильные ли скобки?}
На $i$-м слое определим отношение эквивалентности $\approx_i$: для $u,v\in Y_i$ положим
\begin{center}
$u \approx_i v$ $\iff$ структуры $\str{u}$ и $\str{v}$ изоморфны.
\end{center}
Наконец, положим
$$\sim_i~:=~\sim_{i-1} \cup \approx_i.$$

По индукции легко увидеть, что:
\begin{itemize}
\item если $u\sim_i v$, то $u$ и $v$ лежат в одном слое;
\item если $\h(u)>i$, то $u\sim_i v ~\iff ~u=v$;
\item если $\h(u)<i$, то $u\sim_i v ~\iff ~u\sim_{i-1} v$.
\item если $\h(u)=i$, то $u\sim_i v ~\iff ~u\approx_i v$.
\end{itemize}

\begin{lemma}
$\clB_i$ является правильным разбиением для всех $i\leq h$.
\end{lemma}
\begin{proof}
Индукцией по $i$. Разбиение $\clB_0$ является правильным. Пусть $1\leq i\leq h$.

Пусть $U,V\in \clB_i$ и $u_0 R v_0$ для некоторых $u_0\in U$, $v_0\in V$.
Рассмотрим произвольный $u\in U$ и покажем, что найдется $v\in V$ такой, что $uRv$.

Если $u_0\in X_i$, то $U=\{u_0\}$, то есть $u=u_0$ и доказывать нечего.

Если $u_0\in Z_i$, то $U$ и $V$ лежат в ``верхней'' части шкалы $Z_i$, на которой
$\sim_i$ и $\sim_{i-1}$ совпадают. Поэтому $U,V\in \clB_{i-1}$. В силу предположения индукции, $\clB_{i-1}$ является правильным,
поэтому $uRv$ для некоторого $v\in V$.

Наконец, пусть $u_0\in Y_i$. В этому случае $u\approx_i u_0$, поэтому существует изоморфизм $g: \str{u_0}\to \str{u}$.
Пусть $C$ --- сгусток $u_0$, $D$ --- сгусток $u$.
Поскольку $u_0$ --- точка $i$-го слоя и $u_0 R v_0$, то возможны два случая: $v_0\in Y_i$ и $v_0\in Z_i$.

Первый случай: $v_0\in Y_i$. В этом случае $u_0$ и $v_0$ --- точки одного слоя. Поскольку при этом $u_0 R v_0$, то $u_0$ и $v_0$ лежат в одном сгустке: $v_0\in C$.
Положим $v:=g(v_0)$ и покажем, что $uRv$ и $v\in V$. Имеем: $g(u_0)=u$, $g(v_0)=v$ и $u_0(R\restr C)v_0$, следовательно $u (R\restr D) v$, и поэтому $uRv$.
Поскольку $g$ --- изоморфизм структур $\str{u_0}$ и $\str{u}$, и $g(v_0)=v$, то
$g$ --- изоморфизм структур $\str{v_0}$ и $\str{v}$: действительно, структуры, соответствующие точкам одного сгустка, различаются лишь константой.
Поэтому $v_0\sim_i v$, то есть $v\in V$.

Второй случай: $v_0\in Z_i$.
В этом случае $V\in \clB'_{i-1}$ и сигнатура $\Omega_i$ содержит символ $\un{P}_V$.
Поскольку $u_0 R v_0$, то $u_0\in P^{u_0}_V$. Поскольку $\str{u_0}$ и $\str{u}$ изоморфны, получаем $u\in P^{u}_V$, то есть $uRv$ для некоторого $v\in V$.
\end{proof}

\begin{lemma}
$\clB_i$ является является измельчением $\clA$ для всех $i\leq h$.
\end{lemma}
\begin{proof}
Индукцией по $i$. Разбиение $\clB_0$ является измельчением $\clA$. Пусть $1\leq i\leq h$.

Нужно проверить включение $\sim_i~\subseteq~{\sim_\clA}$.
В силу предположения индукции, $\sim_{i-1}~ \subseteq~{\sim_\clA}$, поэтому достаточно проверить включение $\approx_i ~\subseteq~{\sim_\clA}$.
Пусть $u\approx_i v$. Для некоторого $A\in \clA$ имеем $u\in A$. Тогда $u\in T^u_A$; поскольку $\str{u}$ и $\str{v}$ изоморфны,
то $v \in T^v_A$, следовательно $v\in A$.
\end{proof}

По индукции легко проверить, что все множества $\clB'_0,\dts,\clB'_h$ конечны.
Действительно, $\clB'_0=\emp$, поскольку $Z_1=\emp$. Если $i>0$, то $$\clB'_i=\clB'_{i-1}\cup (Y_i/{\approx_i}).$$
Если конечно $\clB'_{i-1}$, то конечна сигнатура $\Omega_i$, поэтому конечно число неизоморфных структур мощности не более $N$ в этой сигнатуре; следовательно, конечно разбиение $Y_i/{\approx_i}$.

Поскольку $Z_{h+1}=W$, то $\clB'_h=\clB_h$, поэтому $\clB_h$ конечно. Таким образом, $\clB_h$ является правильным конечным измельчением $\clA$ и первое утверждение теоремы доказано.

\medskip

Оценим мощность $\clB_h$.
\begin{lemma} Для $i>0$
\begin{equation}\label{eq:size-of-Y}
|Y_i/{\approx_i}|\leq N^2 \cdot |\clA|^N \cdot 2^{N\cdot |\clB'_{i-1}|+N^2}.
\end{equation}
\end{lemma}
\begin{proof}
Чтобы задать $\Omega_i$"=структуру (с точностью до изоморфизма), нужно: определить мощность её носителя $M$, $1\leq M\leq N$; бинарное отношение (их $2^{M^2}$); $|\clB'_{i-1}|$ подмножеств, интерпретирующих символы $\un{P}_B$, $B\in \clB'_{i-1}$
(это даст нам $2^{M\cdot |\clB'_{i-1}|}$ вариантов); выбрать подмножества, интерпретирующие символы $\un{T}_A$, $A\in \clA$
(несложно увидеть, что это даст нам $|\clA|^M$ вариантов); выбрать константу ($M$ вариантов). Таким образом:
\begin{equation*}
|Y_i/{\approx_i}|\leq N \cdot 2^{N^2}\cdot 2^{N\cdot |\clB'_{i-1}|} \cdot |\clA|^N \cdot N.\qedhere
\end{equation*}
\end{proof}

Оценим теперь мощность $\clB_h$. Для этого оценим мощности $\abs{\clB'_i}$ при $1\leq i\leq h$.
$$
\abs{\clB'_1} = \abs{Y_1/{\approx_1}} \le N^2 \cdot \abs{\clA}^N
\cdot 2^{N\abs{\clB'_0} + N^2} = N^2 \cdot \abs{\clA}^N \cdot 2^{N^2}
\le 2^{(N+2)(\log_2 \abs{\clA}+N)}.
$$
\ISH{Это неверно!?}

По индукции покажем, что для всех $1\leq i\leq h$
$$\abs{\clB'_i} \leq \exp^i_2\bigl( (N+i+1)(\log_2\abs{\clA}+N)\bigr).$$
Используя предположение индукции и то, что $N, \abs{\clA}\geq 1$, получаем:
\begin{align*}
 \abs{\clB'_{i+1}} &\le \abs{\clB'_i} + \abs{Y_{i+1}/{\approx_{i+1}}} \le \\
 \le& \exp^i_2\bigl( (N+i+1)(\log_2\abs{\clA}+N)\bigr) + N^2 \cdot
\abs{\clA}^N \cdot 2^{N\abs{\clB'_i} + N^2} \le\\
 \le& \exp^i_2\bigl( (N+i+1)(\log_2\abs{\clA}+N)\bigr)+ \\
 &\qquad + N^2 \cdot \abs{\clA}^N \cdot 2^{N\cdot \exp^i_2\bigl(
(N+i+1)(\log_2\abs{\clA}+N)\bigr) + N^2} \le\\
 \le&(N^2 \cdot
\abs{\clA}^N +1) \cdot 2^{N\cdot \exp^i_2\bigl(
(N+i+1)(\log_2\abs{\clA}+N)\bigr) + N^2} \le\\
 \le& 2^{2N} \cdot \abs{\clA}^N  \cdot 2^{N\cdot \exp^i_2\bigl(
(N+i+1)(\log_2\abs{\clA}+N)\bigr) + N^2} =\\
 =& 2^{N(\log_2 \abs{\clA} +N+2) + N\cdot \exp^i_2\bigl(
(N+i+1)(\log_2\abs{\clA}+N)\bigr)} \le\\
\le& 2^{\exp^i_2\bigl((N+i+1)(\log_2\abs{\clA}+N)\bigr) + N\cdot
\exp^i_2\bigl((N+i+1)(\log_2\abs{\clA}+N)\bigr)} =\\
=& 2^{(N+1)\cdot \exp^i_2\bigl((N+i+1)(\log_2\abs{\clA}+N)\bigr)} \le\\
 \le& 2^{\exp^i_2\bigl(
(N+i+2)(\log_2\abs{\clA}+N)\bigr)}=\exp^{i+1}_2\bigl(
(N+i+2)(\log_2\abs{\clA}+N)\bigr).
\end{align*}

Таким образом, если $x$ --- мощность начального разбиения $\clA$, то мощность построенного правильного разбиения $\clB_h=\clB_h'$ ограничена
$\exp^h_2\bigl( (N+h+1)(\log_2 x+N)\bigr).$
\end{proof}

\section{Фильтрации в классах $\clF(m,n,h)$ и $\clG(m,h)$}\label{sect:twofam}
Докажем, что для всех $n>m \geq 1,~h\geq 1$ логики классов $\clF(m,n,h)$ и $\clG(m,h)$, а также классов $\clF_*(m,n)$ и $\clG_*(m)$, являются финитно аппроксимируемыми.
Для случая высоты 1 этот результат был известен: фильтрации для $\clG(m,1)$  были построены в \cite{jansana1994};
для шкал $\clF(m,n,1)$ доказательство было получено в \cite{Ku:Shap:TACL:2011}.
С помощью теоремы \ref{thm:finite_depth_part} распространим эти результаты на случай произвольной конечной высоты.
Для этого нужно избавиться от бесконечных сгустков.



\begin{lemma}\label{lemma:finite-clusters-part}
Пусть $\clA$ "--- конечное разбиение $(m,n)$"=шкалы $F$, $n>m\geq 1$.
Тогда существует его измельчение $\clB$ такое, что:
\begin{enumerate}
\item остовы шкал  $F$ и $F_\clB$ изоморфны;
\item все сгустки в $F_\clB$ конечны и имеют мощность не более
$(n-m)|\clA|$;
\item $F_\clB$ является $(m,n)$"=шкалой.
\end{enumerate}
\end{lemma}

Прежде чем мы приступим к построению разбиений $(m,n)$-шкал, приведём несколько вспомогательных наблюдений.

Рассмотрим произвольную шкалу $F=(W,R)$.
Конечная непустая последовательность $w_0  \ldots w_n$ элементов $W$ называется {\em путём в $F$}, если
$w_i R w_{i+1}$ для всех $i<n$.
\emph{Длина пути:} $\l(w_0 \ldots w_n) = n$. Если $w_0=w_n$, то путь называется {\em $w_0$-петлёй} (или просто {\em петлёй}).
В частности, всякая последовательность длины $1$ является петлёй длины 0.
Если конец пути $\alpha$ совпадает с началом пути $\beta$, то {\em соединение путей} $\alpha\circ\beta$ "--- это конкатенация
последовательностей $\alpha'$ и $\beta$, где последовательность $\alpha'$ получена из $\alpha$ удалением последнего элемента; в этом случае $\alpha\circ \beta$ является путём,
длина которого "--- сумма длин путей $\alpha$ и $\beta$.

Для $w \in W$ пусть $(G(w),\circ)$ "--- моноид всех $w$-петель;
его единицей является $w$ "--- петля длины $0$. Для $w$-петли $\alpha$ полагаем, как обычно, $\alpha^0:=w$,
$\alpha^{i+1}:=\alpha\circ\alpha^i$.

Для $k>0$ пусть $g_k(\alpha)$ "--- остаток от деления длины пути $\alpha$ на $k$. Положим
$$G_k(w):=\{g_k(\alpha)\mid \alpha \in G(w)\}.$$

\begin{proposition}\label{prop:subgr}
$G_k(w)$ "--- подгруппа $\ZZ_k$ для любых $k>0$, $w\in W$.
\end{proposition}
\begin{proof}
Если $\alpha, \beta$ "--- $w$-петли, то
$g_k(\alpha)+g_k(\beta)=g_k(\alpha\circ \beta)\in G_k(w)$, $g_k(\alpha^{k-1})$ "--- обратный к $g_k(\alpha)$.
\end{proof}

\begin{proposition}\label{prop:Runiv1}
Если $R^*=W\times W$, то  $$G_k(w)=G_k(v)$$
для любых $k>0$, $w,v\in W$.
\end{proposition}
\begin{proof}
Пусть $z\in G_k(v)$, то есть $z=g_k(\gamma)$ для некоторой $v$-петли $\gamma$. Покажем, что $z\in G_k(w)$.
Пусть $\alpha$ "--- путь из $v$ в $w$, $\beta$ "--- из $w$ в $v$ (такие пути найдутся, поскольку $R^*=W\times W$).
Рассмотрим $w$-петлю $$\delta=\beta\circ (\alpha\circ \beta)^{k-1}\circ \gamma\circ \alpha.$$ Имеем: $\l(\delta)=\l(\gamma)+k\cdot \l(\alpha\circ\beta)$, то есть $g_k(\delta)=z$ и $z\in G_k(w)$.
\end{proof}

Для $d>0$ определим на $W$ отношение $\preceq_d$:
$$
w \preceq_d v \iff \mbox{существует путь из $w$ в $v$, длина которого делится на } d.
$$
Очевидно, что в любой шкале $\preceq_d$ является предпорядком.

\begin{proposition} \label{prop:Runiv3}
Пусть $R^*=W\times W$.
\begin{enumerate}
\item \label{prop:Runiv2-1}
Если $d$ делит длину любой петли, то $\preceq_d$ "--- отношение эквивалентности и
$$|W/{\preceq_d}|\leq d.$$
\item  \label{prop:Runiv3-2}
Для любого $k>0$ существует $d$ такое, что
$d$ делит $k$, $d$ делит длину любой петли
и $$\preceq_d ~\subseteq~\preceq_k.$$
\end{enumerate}
\end{proposition}
\begin{proof}
(1).
Проверим симметричность.
Пусть $\alpha$ "--- путь из $w$ в $v$ такой, что $d$ делит $\l(\alpha)$. Поскольку $R^*=W\times W$, найдется путь $\beta$ из $v$ в $w$.
Путь $\alpha \circ \beta$ является петлей, поэтому $d$ делит $\l(\alpha \circ \beta)$. Тогда $d$ делит и $\l(\beta)$.

Покажем, что $|W/{\preceq_d}|\leq d$. Пусть $W$ содержит хотя бы две различные точки.
Тогда из любой точки существует путь в отличную от неё, поэтому $R$ удовлетворяет условию
$$\AA w \EE v~wRv.$$
Поэтому в шкале существует путь длины $d-1$ $w_1\dts w_d$. Любая точка $w\in W$ эквивалентна одной из точек $w_i$:
действительно, существует путь $\alpha$ из $w_d$ в $w$; тогда $w\preceq_d w_i$ при $i=d-g_d(\alpha)$.

(2).
Рассмотрим какую-нибудь точку $u\in W$ и группу
$G_k(u)$. Если эта группа тривиальна, положим $d:=k$;
в противном случае пусть $d$ "--- наименьший ненулевой элемент $G_k(u)$.
В силу предложения \ref{prop:subgr}, $d$ делит длину любой $u$-петли.
Поэтому в силу предложения \ref{prop:Runiv1} $d$ делит длину любой петли в шкале.

Осталось построить путь кратный $k$ из $w$ в $v$ для случая $w\preceq_d v$. Если $d \neq k$, то существует $v$-петля $\beta$ такая, что $g_k(\beta)=d$.
Поэтому если $\alpha$ "--- путь из $w$ в $v$ кратный $d$, то, поскольку $d$ делит $k$,
длина $\alpha\circ \beta^i$ делится на $k$ при некотором $i$.
\end{proof}


\begin{proposition}\label{prop:filtr-clust}
Рассмотрим шкалу $F=(W,R)$ и эквивалентность на ней $\sim~\subseteq~\sim_R$.
Пусть $\ov{F}=(\ov{W},\ov{R})$ "--- $\sim$-фильтрация шкалы $F$. Тогда:
\begin{enumerate}
\item
$\ov{W}/{\sim_{\ov{R}}}=\{\ov{C} \mid C\in W/{\sim_R}\}$, где $\ov{C}$ "--- $\sim$-классы, лежащие в сгустке $C$;
%
\item
остовы шкал $F$ и $\ov{F}$ изоморфны.
\end{enumerate}
\end{proposition}
\begin{proof}
Пусть $\ov{w}$ обозначает $\sim$-класс точки $w$.

Заметим, что
\begin{equation}\label{eq:R-ovR}
wR^*v ~ \iff ~ \ov{w}{\left(\ov{R}\right)}^*\ov{v}.
\end{equation}
Действительно, если $wR^i v$, то $\ov{w}{\left(\ov{R}\right)}^i\ov{v}$ по определению $\ov{R}$;
обратная импликация имеет место в силу того, что $\sim~\subseteq~\sim_R$.
Из этого следует, что
\begin{equation}\label{eq:simR-simovR}
w\sim_Rv \iff \ov{w}\sim_{\ov{R}}\ov{v}.
\end{equation}
Поэтому если $X$ "--- сгусток в $\ov{F}$, то $C=\cup X$ "--- сгусток в $F$  и
$X=\ov{C}$. Если же $C$ "--- сгусток в $F$, то, опять же в силу (\ref{eq:simR-simovR}), $\ov{C}$ "--- сгусток в $\ov{F}$.

Отображение $C\mapsto \ov{C}$ является изоморфизмом между остовами шкал $F$ и $\ov{F}$ в силу (\ref{eq:R-ovR}).
\end{proof}

\begin{proof}[Доказательство леммы \ref{lemma:finite-clusters-part}]
Пусть $\set{C_i}_{i \in I}$ "--- все сгустки шкалы $F=(W,R)$, то есть классы эквивалентности по отношению $\sim_R$.
Положим $R_i=R\restr C_i$. Заметим, что $R_i^*=C_i\times C_i$. Положим $k=n-m$.
В силу предложения \ref{prop:Runiv3}, на $C_i$ существует отношение эквивалентности $\sim_i$ такое, что
\begin{center}
если $w\sim_i v$,  то $k$ делит длину некоторого пути из $w$  в $v$, и
\end{center}
\begin{equation}\label{eq:prof-fin-cl:2}
|C_i/{\sim_i}|\leq k.
\end{equation}
На каждом сгустке $C_i$ определим эквивалентность
$$\approx_i~:=~\sim_i\cap  \sim_{\clA}.$$
В силу (\ref{eq:prof-fin-cl:2}), имеем
\begin{equation}\label{eq:prof-fin-cl:3}
|C_i/{\approx_i}|\leq k|\clA|.
\end{equation}
Определим на $W$ отношение эквивалентности $\approx$ и соответствующее ему разбиение $\clB$:
$$\approx~:=~\bigcup_{i\in I}\approx_i,~~\clB:=W/{\approx}.$$
Поскольку $\approx~\subseteq~\sim_{\clA}$, то $\clB$ "--- измельчение $\clA$.
Кроме того, $\approx~\subseteq~\sim_R$. Поэтому
в силу предложения \ref{prop:filtr-clust}
остовы шкал $F$ и $F_\clB$ изоморфны, а мощности сгустков последней ограничены  $k|\clA|$.

\medskip
Осталось показать, что $F_\clB$ является $(m,n)$-шкалой. 

Заметим, что если  $[u]_\approx ~R_\clB~[v]_\approx$, то $u R^{zk+1}v$ для некоторого $z\geq 0$:
действительно, в силу определения $R_\clB$, найдутся $u'\approx u$, $v' \approx v$ такие, что $u' R v'$;
поскольку $u'\sim_i u$ и $u'\sim_j u$ для некоторых $i,j\in I$, то найдутся пути $\alpha$ из $u$ в $u'$ и
$\beta$ из $v$ в $v'$, длины которых делятся на $k$.

Пусть $[u]_\approx \left(R_\clB\right)^n [v]_\approx.$ Тогда
$u R^{zk+n} v$ для некоторого $z\geq 0$. Поскольку $F$ является $(m,n)$-шкалой,
$R^{zk+n}\subseteq \dots\subseteq R^n \subseteq R^m$.
Поэтому $u R^{m} v$. Следовательно,
$[u]_\approx \left(R_\clB\right)^m [v]_\approx$.
\end{proof}

Сформулируем аналог доказанной леммы для $m$-транзитивных шкал.

\begin{lemma}\label{lemma:finite-clustersMtrans-sat}
Пусть $\clA$ "--- конечное разбиение $m$-транзитивной $F$, $m\geq 1$ \ISH{$m\geq 1$? А почему не 0? Отследить по всей работе}.
Тогда существует его измельчение $\clB$ такое, что:
\begin{enumerate}
\item остовы шкал  $F$ и $F_\clB$ изоморфны;
\item все сгустки в $F_\clB$ конечны и имеют мощность не более
$|\clA|$;
\item $F_\clB$ $m$-транзитивна.
\end{enumerate}
\end{lemma}
\begin{proof}
Доказательство этого факта проще, чем доказательство предыдущей леммы: при фильтрации нам не нужно строить
путь данной длины "--- достаточно показать, что любой путь длины более $m$ можно сократить до пути длины не превосходящей $m$.

$\set{C_i}_{i \in I}$ и $R_i$
определим так же, как и в доказательстве леммы \ref{lemma:finite-clusters-part}.
Разбиение $\clB$ определяется проще:
$$\approx~:=~\sim_{\clA}\cap \sim_R,~\clB:=W/{\approx}.$$
В силу тех же аргументов, что и в доказательстве леммы \ref{lemma:finite-clusters-part},
$\clB$ является измельчением $\clA$,
остовы шкал $F$ и $F_\clB$ изоморфны и мощности сгустков последней ограничены  $|\clA|$.

Осталось показать, что $F_\clB$ $m$-транзитивна. Так как $\approx ~\subseteq ~\sim_R$, то из
$[u]_\approx ~R_\clB~ [v]_\approx$ следует $u R^{z+1}v$ для некоторого $z\geq 0$. Поэтому если
$[u]_\approx \left(R_\clB\right)^{m+1} [v]_\approx,$ то
$u R^{m+1+z} v$ для некоторого $z\geq 0$. Поскольку $F$ $m$-транзитивна, то $u R^{z} v$ для некоторого $z\leq m$; следовательно,
$[u]_\approx \left(R_\clB\right)^z [v]_\approx$.
\end{proof}

\begin{theorem}\label{thm:fmp-two-classes} Пусть $n>m \geq 1,~h\geq 1$.
Логики классов $$\clF(m,n,h),~ \clF_*(m,n),~\clG(m,h),~\clG_*(m)$$   финитно аппроксимируемы.
\end{theorem}
\begin{proof}
Пусть $\vf$ выполнима в некоторой шкале $F\in\clF_*(m,n)$, $h$ "--- высота $F$. Тогда $\vf$ истинна в одной из точек некоторой модели $M$ над $F$.
Рассмотрим на $M$ эквивалентность $\sim_\vf$. Пусть $\clA$ "--- классы эквивалентности $\sim_\vf$. Применив лемму
\ref{lemma:finite-clusters-part}, мы получим шкалу $G\in \clF_*(m,n)$ конечной высоты $h$, в которой также выполнима формула $\vf$. Причём в $G$ все сгустки конечны и ограничены в совокупности, поэтому, применив теорему \ref{thm:finite_depth_part} к разбиению $\clB$, соответствующему эквивалентности $\sim_\vf$ в подходящей модели над $G$, получим правильное разбиение $\clC$ шкалы $G$ такое, что $\vf$ выполнима в $G_\clC$.
Финитная аппроксимируемость логик классов $\clF(m,n,h)$ и $\clF_*(m,n)$ теперь следует из предложения \ref{prop:franz-pmorph}.

Доказательство теоремы для классов $\clG(m,h)$ и $\clG_*(m)$ совершенно аналогично, нужно лишь вместо леммы \ref{lemma:finite-clusters-part}
применить лемму \ref{lemma:finite-clustersMtrans-sat}.
\end{proof}

\section{Модальная аксиоматизация}\label{sect:ax}
\subsection{Определимость}

Положим $\Box\vf=\neg\Di\neg \vf$,
\begin{align*}
\Di^0 \vf &= \vf, &\Di^{i+1}\vf&=\Di \Di^i \vf, &\Di^{\leq m} \vf &= \bigvee_{i=0}^{m} \Di^i \vf,\\
\Box^0\vf &=\vf,   &\Box^{i+1}\vf&=\Box\Box^i\vf, &\Box^{\leq m} \vf
&= \bigwedge_{i=0}^{m} \Box^i \vf.
\end{align*}
Легко проверить следующий факт.
\begin{proposition}\label{prop:mn-def}~
\begin{itemize}
\item
$(W,R)$ "--- $m$-транзитивная шкала $~\iff~$ $(W,R)\mo \Di^{m+1} p \imp \Di^{\leq m} p$.
\item
$(W,R)$ "--- $(m,n)$-шкала $~\iff~$ $(W,R)\mo \Di^n p \imp \Di^m p$.
\end{itemize}
\end{proposition}

Известно, что для предпорядков свойство $\h(F)\leq h$ можно выразить модальной формулой, см., например, \cite[Proposition 3.44]{Ch:Za:ML:1997}.
А именно, определим по индукции формулы $B_h$, $h>0$:
$$
 B_1 = p_1 \to \Box\Di p_1, ~~ B_{h+1} = p_{h+1} \to  \Box(\Di p_{h+1} \lor B_h).
$$
Для любого предпорядка $F$,
$$F\mo B_h \iff \h(F)\leq h.$$
Опишем аналогичные формулы для $m$-транзитивных шкал.
Пусть $B_h(m)$ обозначает формулу $B_h$, записанную через операторы
$\Box^{\leq m}$ и $\Di^{\leq m}$:
$$
B_1(m) = p_1 \to \Box^{\leq m}\Di^{\leq m} p_1,~~ B_{h+1}(m) = p_{h+1} \to  \Box^{\leq m}(\Di^{\leq m} p_{h+1} \lor B_h(m)).
$$
Очевидно, что в любой модели над $m$-транзитивной шкалой имеем
\[
\begin{array}{lll}
M,x\mo \Di^{\leq m} \vf  &\iff &  \EE y (xR^*y \textrm{ и } M,y\mo \vf),\\
M,x\mo \Box^{\leq m} \vf  &\iff &  \AA y (xR^*y \Imp M,y\mo \vf).
\end{array}
\]
Из этого следует
\begin{proposition}\label{prop:height}
Пусть $F$ "--- $m$-транзитивная шкала. Тогда
$$F\mo B_h(m)  \iff ~ \h(F)\leq h.$$
\end{proposition}

Класс шкал $\clF$ {\em модально определим множеством формул} $\Phi$, если
$$F\in \clF ~\iff~ F\mo \Phi.$$

\begin{proposition}\label{prop:twoclasses-defined}
Пусть $m,n,h\geq 1$.
Класс $\clG(m,h)$ модально определим формулами $\{\Di^{m+1} p \imp \Di^{\leq m} p, ~B_h(m)\}$.
При $n>m$ класс $\clF(m,n,h)$ модально определим формулами $\{\Di^n p \imp \Di^m p, ~B_h(n-1)\}$.
\end{proposition}
\begin{proof}
Из предложений \ref{prop:mn-def} и \ref{prop:height}.
\end{proof}

\subsection{Предтранзитивные логики}
Напомним, что множество модальных формул $\Lvar$
называется {\em логикой}
(точнее "--- {\em пропозициональной нормальной модальной логикой}), если
\begin{itemize}
\item $\Lvar$ содержит все пропозициональные тавтологии,
\item формулы $\neg\Di \bot$ и $\Di (p_1\vee p_2)\imp\Di  p_1\vee \Di  p_2$
\item и замкнуто относительно правила Modus Ponens,
правила подстановки и {\em правила монотонности}:
\[
\vf\imp \psi\in \Lvar \Imp \Di \vf \imp \Di \psi\in \Lvar.
\]\end{itemize}
Наименьшая логика обозначается $\LK{}$.
Если $\Lvar$ "--- логика, $\Phi$ "--- множество формул, то
$\Lvar+\Phi$ обозначает наименьшую логику, содержащую $\Lvar\cup\Phi$.

Логика всякого класса шкал является логикой в смысле этого определения; такая логика называется {\em полной по Крипке}.
Напомним \cite{Ch:Za:ML:1997}, что $\LS{4}=\LK{}+\{p\imp \Di p, \Di \Di p\imp \Di p\}$ "--- логика класса всех (конечных) предпорядков,
$\LS{5}=\LS{4}+\{p\imp \Box \Di p\}$ "--- логика класса всех (конечных) шкал, в которых отношение является эквивалентностью.

Из теоремы Салквиста (см., например, \cite[Theorem 10.30]{Ch:Za:ML:1997}) следует
\begin{proposition}\label{prop:sahlq-compl} Для всех $m,n\geq 0$,
$\Log \clF(m,n) =\LK{}+\{\Di^{n} p \imp \Di^{m} p\}$,
$\Log \clG(m) =\LK{}+\{\Di^{m+1} p \imp \Di^{\leq m} p\}$.
\end{proposition}

\begin{definition}
Логика называется {\em предтранзитивной}, если для некоторого $m\geq 0$ она содержит {\em формулу $m$-транзитивности} $\Di^{m+1} p \imp \Di^{\leq m} p$;
$\trm(\Lvar)$ обозначает наименьшее такое $m$.
\end{definition}
В силу предложения \ref{prop:mn-def}, предтранзитивными являются все логики $\Log \clG(m)$, $m\geq 0$.

\begin{proposition} Пусть $m,n\geq 0$, $\Lvar=\Log \clF(m,n)$.
\begin{itemize}
\item При $n>m$ $\Lvar$ предтранзитивна и $\trm(\Lvar)= n-1$.
\item При $n\leq m$ $\Lvar$ не предтранзитивна.
\end{itemize}
\end{proposition}
\begin{proof}
Если $(W,R)\in \clF(m,n)$ и $n>m$, то $R^n\subseteq R^{\leq n-1}$. 
Поэтому $\Lvar$ предтранзитивна  и $\trm(\Lvar)\leq n-1$.
С другой стороны, шкала $(\{1,\dts,n\},R)$, где $iRj \iff j=i+1$, является $(m,n)$-шкалой, поскольку $R^n=\emp$. Однако она не является $l$-транзитивной ни для какого $l<n-1$.

Чтобы доказать второй пункт, рассмотрим шкалу $(\mathbb{N},R)$,
где $$R=\{(i,i+1)\mid i\in \mathbb{N}\}\cup \{(i,i)\mid i\in \mathbb{N}\}.$$
В этой шкале пути можно произвольно удлинять, но нельзя сокращать, то есть эта шкала яляется
$(m,n)$-шкалой для всех $n<m$, но ни для какого $l$ не является $l$-транзитивной.
\end{proof}

\hide{
[Оно надо? ]
\begin{proposition}\cite{ShefSkvGab} ~\label{lem:pretr}
Логика $\Lvar$ предтранзитивна тогда и только тогда,  когда существует формула $\chi(p)$ с одной переменной $p$, такая, что в любой модели
Крипке $M$, если \/ $M\mo \Lvar$, то для любой точки $w$ в $M$
\[
M,x\mo\chi(p) \iff M^x\mo p.
\]
\end{proposition}
Для $m$-транзитивной логики роль $\chi(p)$ [...почему] выполняет формула $\Di^{\leq m} p$.

}

\subsection{Теорема Гливенко и финитная аппроксимируемость}
Теорема Гливенко утверждает, что выводимость формулы $\vf$ в классической логике высказываний равносильна выводимости $\neg\neg\vf$ в интуиционистской логике. Аналогичная сводимость имеет место для $\LS{4}$ и $\LS{5}$ \cite{Mats-s4s5}:
$$\vf\in \LS{5}~\iff~\Di\Box \vf\in \LS{4}.$$
В семантике Крипке интуиционистская логика "--- это логика частичных порядков, а классическая "--- это логика частичных порядков высоты 1  (см., например, \cite{Ch:Za:ML:1997}).
Аналогично,  $\LS{4}$ "--- это логика предпорядков, $\LS{5}$ "--- логика отношений эквивалентности, то есть предпорядков высоты 1.
Сформулируем  аналог теоремы Гливенко для предтранзитивных логик.

Пусть $\Lvar$ "--- предтранзитивная логика, $m=\trm(\Lvar)$. Положим $$\Lvar[h]=\Lvar+\{B_h(m)\}.$$
Положим также $\Di^* \vf=\Dim\vf $, $\Box^* \vf=\Boxm \vf$.
\begin{theorem}\label{thm:Glivenko} Если $\Lvar$ "--- предтранзитивная логика, то
$$\vf\in \Lvar[1] ~\iff~ \Di^* \Box^* \vf\in \Lvar.$$
\end{theorem}
\begin{proof}
Для формулы $\psi$ пусть $\psi^*$ обозначает результат
замены в $\psi$ $\Di$ на $\Di^*$ (соответственно, $\Box$ заменяется на $\Box^*$).
Известно, что множество $\{\psi \mid \psi^* \in \Lvar\}$ является логикой и содержит $\LS{4}$ (см., например, \cite{ShefSkvGab}).
\ISH{Проверить, уточнить ссылку}.

Пусть $\Di^*\Box^*\vf \in \Lvar$. Тогда $\Di^*\Box^*\vf \in \Lvar[1]$.
Поскольку $p\imp \Box^* \Di^*p \in \Lvar[1]$, то $\Di^*\Box^* p \imp p\in\Lvar[1]$. Следовательно,  $\Di^*\Box^* \vf \imp \vf\in \Lvar[1]$ и
$\vf\in \Lvar[1]$.

Доказательство в обратную сторону проведём индукцией по
длине вывода формулы $\vf$ в $\Lvar[1]$.

Допустим, что $\vf=p\imp \Box^*\Di^*p$.
Формула $\Di\Box(p\imp \Box \Di p)$ общезначима на любом конечном предпорядке, поэтому $\Di\Box(p\imp \Box \Di p)\in \LS{4}$,
и поэтому $\Di^*\Box^*\vf\in \Lvar$.

Случай, когда $\vf$ получена в результате применения правила подстановки, тривиален. Рассмотрим случаи применения правил
Modus Ponens и монотонности.

Пусть $\psi, ~ \psi  \imp\vf \in \Lvar[1]$ для некоторой $\psi$. По
предположению индукции, $\Di^*\Box^*\psi ,Ё\Di^*\Box^*(\psi \imp \vf) \in\Lvar$.
Тогда $\Box^*\Di^*\Box^*(\psi \imp
\vf)\in \Lvar$ (\cite[Lemma 1.3.45]{ShefSkvGab}). Формула $\Di\Box p\wedge
\Box\Di\Box (p\imp q) \imp \Di\Box q$ выводима в $\LS{4}$, поскольку
общезначима в любом предпорядке. Поэтому $\Di^*\Box^* \vf\in \Lvar$.

Предположим, что $\vf=\Di \psi_1\imp \Di\psi_2$, $\psi_1\imp \psi_2 \in \Lvar[1]$.
Покажем, что $\Di^*\Box^*(\Di \psi_1\imp \Di \psi_2) \in \Lvar$.
По предположению индукции,
$\Di^*\Box^*(\psi_1\imp \psi_2) \in \Lvar$.
Пусть $m=\trm(\Lvar)$.
В силу предложения \ref{prop:sahlq-compl}, $\Log G(m) \subseteq \Lvar$.
Формула $$\Dim\Boxm(\psi_1\imp \psi_2)\imp \Dim \Boxm(\Di\psi_1\imp \Di \psi_2)$$
общезначима в любой $m$-транзитивной шкале, поэтому принадлежит $\Lvar$. Следовательно,
$\Dim\Boxm \vf\in \Lvar$.
\end{proof}

Напомним понятие порождённой подмодели, см. \cite{Ch:Za:ML:1997}.
Пусть $F=(W,R)$ "--- шкала, $M=(F,\theta)$ "--- модель.
\emph{Сужением} $M$
на $V\neq \emp$ называется модель $M\restr V=(F\restr V, \theta')$, где
$\theta'(p)=\theta(p)\cap V$ для всех  $p\in PV$.
Для $V\subseteq W$, $R(V)=\{y\mid \EE x\in V ~xRy\}$  "--- {\em образ $V$ относительно $R$}.
Если $R(V)\subseteq V$,
то $F\restr V$ и $M\restr V$ называются {\em порождённой подшкалой и подмоделью $F$ и
$M$}, соответственно. Имеет место следующий факт.
\begin{proposition}[Лемма о порождённой подмодели]\label{lem:gen-sub}
Пусть $F=(W,R)$, $M=(F,\theta)$, $V$ "--- непустое подмножество $W$,
$R(V)\subseteq V$. Тогда:
\begin{enumerate}
\item если $x\in V$, то  $M,x\mo\vf~\iff~ M\restr V,x\mo\vf;$
\item если $F\mo \vf$, то $F\restr V \mo \vf$.
\end{enumerate}
\end{proposition}

Теорема \ref{thm:Glivenko}  даёт необходимое условие разрешимости и финитной аппроксимируемости предтранзитивных логик.
\begin{corollary}
Если $\Lvar$ разрешима, то $\Lvar[1]$ разрешима.
Если $\Lvar$ финитно аппроксимируема, то $\Lvar[1]$ финитно аппроксимируема.
\end{corollary}
\begin{proof}
Первое утверждение немедленно следует из теоремы \ref{thm:Glivenko}.

Пусть $\Lvar$ финитно аппроксимируема. Покажем, что $\Lvar[1]$ "--- логика класса всех конечных $\Lvar$-шкал высоты $1$.
Пусть $\vf\notin \Lvar[1]$. Тогда $\Di^*\Box^*\vf \notin \Lvar$ в силу теоремы \ref{thm:Glivenko}.
Тогда $\Box^*\Di^*\neg \vf$ истинна в одной из точек некоторой модели $M$ над какой-то конечной $\Lvar$-шкалой $F$. Тогда
$\neg\vf$ истинна в некоторой точке $x$ одного из максимальных сгустков $C$ модели $M$. По лемме о порождённой подмодели,
$M\restr C,x\mo \neg\vf$ и $F\restr C$ является $\Lvar$-шкалой. Осталось заметить, что $F\restr C$ "--- шкала высоты $1$.
\end{proof}

\begin{proposition}\label{prop:consis}
Пусть $\Lvar$ "--- предтранзитивная логика. Тогда:
\begin{enumerate}
\item
$\Lvar[1]~\supseteq~\Lvar[2]~\supseteq~\Lvar[3]~\supseteq~\dts~\supseteq~\Lvar.$
\item
Если $\Lvar$ непротиворечива, то $\Lvar[1]$ непротиворечива
(и, следовательно, непротиворечивы все логики $\Lvar[h]$, $h\geq 1$).
\end{enumerate}
\end{proposition}
\begin{proof}
Включение $L~\subseteq~\Lvar[h]$ очевидно.
Проверим включение
$\Lvar[h+1]~\subseteq~\Lvar[h]$. Пусть $m=\trm(\Lvar)$.
$B_h(m)\in \Lvar[h]$, следовательно $\Dim p_{h+1}\vee B_h(m)\in \Lvar[h]$.
Тогда $\Boxm(\Dim p_{h+1}\vee B_h(m))\in \Lvar[h]$, из чего получаем
$$B_{h+1}(m)=p_{h+1}\imp\Boxm(\Dim p_{h+1}\vee B_h(m))\in \Lvar[h].$$

Пусть $\Lvar$ непротиворечива. Докажем непротиворечивость $\Lvar[1]$.
Формула $\neg \Di \Box\bot$ общезначима в любом частичном
порядке, поэтому принадлежит $\LS{4}$. Следовательно, $\neg\Dim\Boxm\bot\in \Lvar$.
Поэтому  $\Dim\Boxm\bot\notin \Lvar$. В силу теоремы \ref{thm:Glivenko}, $\bot\notin \Lvar[1]$.
\end{proof}

\subsection{Полнота по Крипке логик конечной высоты}
Приведём одно достаточное условие полноты по Крипке логик, содержащих формулы конечной высоты.

Напомним определение {\em канонической модели непротиворечивой
логики} $\Lvar$. Множество
формул $\Gamma$ называется {\em $\Lvar$-противоречивым}, если $\neg
\bigwedge \Gamma_0 \in \Lvar$ для некоторого конечного
$\Gamma_0\subseteq \Gamma$. $\Gamma$ называется {\em
$\Lvar$-максимальным}, если оно $\Lvar$-непротиворечиво, а любое его
собственное расширение оказывается $\Lvar$-противоречивым. Известно, что
всякое $\Lvar$-непротиворечивое множество содержится в $\Lvar$-максимальном
(лемма Линденбаума). {\em Канонической шкалой} непротиворечивой логики $\Lvar$ называется
шкала $F_\Lvar=(W_\Lvar,R_\Lvar)$, где $W_\Lvar$ "--- множество всех $\Lvar$-максимальных
множеств, а $R_\Lvar$ определяется следующим образом:
$$\Gamma_1 R_\Lvar \Gamma_2 \iff \{ \Di \vf \mid \vf\in \Gamma_2\} \subseteq \Gamma_1.$$
{\em Каноническая модель} $M_\Lvar$ "--- это каноническая шкала с оценкой
$\theta_\Lvar$, где $\theta_\Lvar(p):=\{\Gamma\in W_\Lvar\mid p\in \Gamma\}$.
Известен следующий важный факт (см., например, \cite{blackburn_modal_2002}):
\begin{proposition}[Теорема о канонической модели]\label{prop:canon}
Если $\Lvar$ "--- непротиворечивая логика, то
\begin{enumerate}
\item
$\AA \Gamma \in W_\Lvar  (M_\Lvar,\Gamma\mo \vf~\iff~\vf\in\Gamma)$,
\item
$\Lvar =\{\vf \mid  M_\Lvar\mo \vf\}$.
\end{enumerate}
\end{proposition}
Логика называется {\em канонической}, если она общезначима в своей канонической шкале. В этом случае она полна по Крипке.

\begin{proposition}\label{prop:canon-R-trans}
Пусть $F=(W,R)$ "--- каноническая шкала логики $\Lvar$.
\begin{enumerate}
\item Для любого $i\geq 0$,
$$ x R^i y \iff \AA\vf (\vf\in y \Imp \Di^i \vf\in x).$$
\item
Если $\Lvar$ $m$-транзитивна, то
 $$x R^* y \iff \AA\vf (\vf\in y \Imp \Di^{\leq m} \vf\in x).$$
\end{enumerate}
\end{proposition}
\begin{proof} Доказательство первого пункта можно найти в \cite[Corollary 5.10]{Ch:Za:ML:1997}. Проверим второй.

Необходимость. В силу теоремы Салквиста, $F\mo \Di^{m+1} p\imp \Di^{\leq m} p$, поэтому из $xR^*y$ следует
$xR^iy$ для некоторого $i\leq m$. Если в этом случае $\vf\in y$, то, в силу пункта 1, $\Di^i\vf\in x$, и поэтому $\Dim\vf\in x$.

Достаточность.
Предположим, что $x R^* y$ неверно.
Тогда для всех $i\leq m$ неверно, что $x R^i y$. В силу пункта 1,  для каждого $i\leq m$
найдётся формула $\vf_i \in y$ такая, что $\Di^i\vf_i \nin x$.
Пусть $\vf=\vf_0\con\dts\con \vf_m$. Имеем: $\vf\in
y$ и для всех $i\leq m$ $\Di^i\vf  \not\in x$.
Тогда $x$ содержит формулу \mbox{$\neg \vf  \con \neg\Di \vf  \con \dts \con
\neg \Di^m \vf$}. Последнее означает, что $\Dim\vf\notin x$.
\end{proof}

Следующий факт следует непосредственно из определений.
\begin{proposition}\label{prop:canon-sublogic}
Если $\Lvar_1$ и $\Lvar_2$ "--- непротиворечивые логики и $\Lvar_2\subseteq \Lvar_1$,
то каноническая модель $\Lvar_1$ является порождённой подмоделью
канонической модели $\Lvar_2$ и состоит из всех
$\Lvar_2$-максимальных множеств, содержащих $\Lvar_1$.
\end{proposition}

Через $\cone{F}{x}$ обозначается шкала (модель),
\emph{порождённая точкой $x$}: $\cone{F}{x}=F\restr \{y \mid x  R^* y\}$, где, как обычно, $R^*$ "--- транзитивное рефлексивное замыкание отношения в шкале.


Полнота по Крипке и каноничность расширений $\LS{4}$ формулами $B_h$ известна достаточно давно \cite{Seg_Essay,Max1975}.  Этот факт имеет следующее обобщение.
\begin{theorem}\label{thm:compl}
Если $L$ предтранзитивная каноническая логика, то для всех $h\geq 1$
канонической оказывается логика $\Lvar[h]$.
\end{theorem}
\begin{proof}
Пусть $M_h=(W_h,R_h,\theta_h)$ обозначает каноническую модель логики
$\Lvar[h]$, $M=(W,R,\theta)$ "--- каноническую модель $\Lvar$.

В силу предложений
\ref{prop:consis} и \ref{prop:canon-sublogic},
$$
M_1\subgen M_2 \subgen M_3\subgen \dts \subgen M,
$$
где $N \subgen M$ означает, что $N$ порождённая подмодель $M$.

Пусть $F$ "--- каноническая шкала логики $L$.
\begin{lemma}
\begin{equation}\label{eq:depth}
M,x\mo \Lvar[h] \iff \h(F,x)\leq h.
\end{equation}
\end{lemma}
\begin{proof}
Если $\h(F,x)\leq h$, то   $\cone{F}{x}\mo \Lvar[h]$ в силу предложения
\ref{prop:height}; следовательно, $\cone{M}{x},x\mo \Lvar[h]$; в силу леммы о
порождённой подмодели, $M,x\mo \Lvar[h].$

Доказательство того, что  из $M,x\mo \Lvar[h]$ следует $\h(F,x)\leq h$
проведём индукцией по высоте $h$.

Пусть $m=\trm(\Lvar)$.

Предположим $M,x\mo \Lvar[1]$.
Тогда в силу предложения \ref{prop:canon-sublogic}, $x\in W_1$. $B_1(m)$ является формулой Салквиста, следовательно $(W_1,R_1)\mo B_1(m)$, что, в силу
предложения  \ref{prop:height}, означает $\h(W_1,R_1)=1$.
Следовательно,
$\h(F,x)=1$.

Пусть теперь $M,x\mo \Lvar[h+1]$. Если $\h(F,x)=1$, то доказывать
нечего. В противном случае в модели $M_{h+1}$ найдется
$y$ такой, что $[x]<_R [y]$ (тут $[x]$ обозначает сгусток точки $x$ в шкале $(W,R)$, или,
что эквивалентно, в шкале $(W_{h+1},R_{h+1})$).
Рассмотрим модель $\cone{M}{y}$ и покажем, что $\cone{M}{y}\mo \Lvar[h]$. Для этого
достаточно показать, что в любой точке $z$ модели $\cone{M}{y}$ истинны все
подстановки в формулу $B_h(m)$. Пусть $\vf=B_h(m)(\psi_1,\dts,\psi_n)$.
Поскольку $[x]<_R [y]$ и $[y]\leq_R [z]$, то неверно, что $[z]\leq_R
[x]$. Следовательно, неверно,
что $z R^* x$. В силу предложения \ref{prop:canon-R-trans},
последнее означает, что существует формула $\psi$ такая, что
$\psi\in x$ и $\Dim\psi\nin z$. Формула
$$\alpha=\psi\imp \Boxm(\Dim \psi \vee \vf)$$
является результатом подстановки в формулу $B_{h+1}(m)$, и поэтому
$M,x\mo \alpha$. Следовательно,
$$M,x\mo \Boxm(\Dim \psi \vee \vf).$$ Из этого следует, что $$M,z\mo \Dim \psi \vee \vf.$$
Поскольку $M,z\not\mo \Dim \psi$, то $M,z\mo \vf$, что нам и
требовалось.

Таким образом, $\cone{M}{y}\mo \Lvar[h]$ и, в силу предположения индукции,
$\h(F,y)\leq h$. Из этого следует, что $\h(F,x)\leq h+1$.
\end{proof}

В силу предложения \ref{prop:canon-sublogic}, $M,x\mo \Lvar[h] \iff x\in
W_h$, что, в силу только что доказанной леммы, даёт $x\in W_h\iff
\h(F,x)\leq h$. Это означает, что высота шкалы $(W_h,R_h)$  не
превосходит $h$. В силу
предложения \ref{prop:height}, $(W_h,R_h)\mo \Lvar[h]$.
\end{proof}

\begin{corollary}
Для $n>m \geq 1,~h\geq 1$
\[
\begin{array}{lll}
\Log \clG(m,h) &=&K+\{\Di^{m+1} p \imp \Di^{\leq m} p, ~B_h(m)\},  ~~\\
\Log \clF(m,n,h) &=&K+\{\Di^{n} p \imp \Di^{m} p, ~B_h(n-1)\}.
\end{array}
\]
\end{corollary}
\begin{proof}
В силу теоремы Салквиста, логики $K+\{\Di^{m+1} p \imp \Di^{\leq m} p\}$ и
$K+\{\Di^{n} p \imp \Di^{m} p\}$ являются каноническими. В силу теоремы \ref{thm:compl}, их расширения формулами конечной высоты являются полными по Крипке.
По предложению \ref{prop:twoclasses-defined},  $\clG(m,h)$ "--- класс всех шкал первой логики,  $\clF(m,n,h)$ "--- второй.
\end{proof}

Для конечно аксиоматизируемой логики её финитная аппроксимируемость даёт её разрешимость (теорема Харропа).
\begin{theorem}\label{thm:complete-ax-twofam}
Для  $n>m \geq 1,~h\geq 1$,  логики классов $\clF(m,n,h)$ и $\clG(m,h)$  разрешимы.
\end{theorem}
\begin{proof}
Логики этих классов задаются конечным множеством аксиом и, в силу теоремы \ref{thm:fmp-two-classes}, являются финитно аппроксимируемыми.
\end{proof}

\section{Некоторые следствия}\label{sect:cor}
Конечность шкалы равносильна тому, что конечны её высота, все слои её остова и число точек в каждом из её сгустков.
При условии, что предтранзитивная $(m,n)$-шкала имеет конечную высоту,
для данной выполнимой формулы
остальные параметры позволяют ограничить теорема \ref{thm:finite_depth_part} и лемма
\ref{lemma:finite-clusters-part}.  
\begin{corollary*} Пусть $n > m\geq 1$.
\begin{itemize}
\item $\Log\clF(m,n)$ финитно аппроксимируема тогда и только тогда, когда она совпадает с $\Log\clF_*(m,n)$.
\item $\Log\clG(m)$ финитно аппроксимируема тогда и только тогда, когда она совпадает с $\Log\clG_*(m)$.
\end{itemize}
\end{corollary*}

Какие свойства логик $\Log\clF_*(m,n)$ и $\Log\clG_*(m)$ мы можем указать, кроме установленной финитной аппроксимируемости?
В частности, обладают ли они конечной аксиоматикой и являются ли они разрешимыми?

\bigskip
\ISH{Шеф рекомендовал ввести понятие эффективной ФА. Тут есть тонкость --- ниже говорится об эффективности не для всех шкал логики, а для класса}
Финитная аппроксимируемость может дать разрешимость модальной логики и без установления её полной конечной аксиоматики: $\Log \clF$ разрешима,
если существует вычислимая функция $f$ такая, что
любая выполнимая в $\clF$ формула $\vf$ выполнима в шкале из $\clF$, размер которой не превышает $f(\l(\vf))$,
и проблема принадлежности конечной шкалы классу $\clF$ разрешима.  Последнее имеет место, если, например,  класс шкал задаётся конечным множеством модальных формул или предложений первого порядка. Ограничение на размер шкал для классов, допускающих $f$-ограниченные минимальные фильтрации, даёт предложение \ref{prop:admitsfiltr-fmp}. Найдем такие ограничения для классов $\clF(m,n,h)$ и $\clG(m,h)$, $n>m\geq 0$, $h\geq 1$.

\begin{proposition}\label{prop:filtr-compose}
Пусть $F=(W,R)$ --- шкала, $\clA$ --- разбиение $W$, $\clB$ --- разбиение $\clA$ и
$\clC=\{\cup B\mid B\in \clB\}$.
Тогда шкалы
$(F_\clA)_\clB$ и $F_\clC$ изоморфны.
\end{proposition}
\begin{proof}
Очевидно, что $\clC$ --- разбиение $W$, и отображение $B\mapsto \cup B$ ---
биекция между $\clB$ и $\clC$.


По определению, $F_\clA=(\clA,R_\clA)$, $(F_\clA)_\clB=(\clB,(R_\clA)_\clB)$, $F_\clC=(\clC,R_\clC)$. Проверим формально, что для любых $B_1,B_2\in\clB$ имеем
$B_1~(R_\clA)_\clB~B_2 \iff (\cup B_1) ~R_\clC ~(\cup B_2)$:

\begin{align*}
B_1~ (R_\clA)_\clB~B_2 &~\iff~ \EE A_1\in B_1 ~\EE A_2\in B_2 ~ A_1 R_\clA A_2 \\
&~\iff~ \EE A_1\in B_1 ~\EE A_2\in B_2 ~ (\EE x_1\in A_1~ \EE x_2\in A_2~  x_1 R x_2) \\
&~\iff~ \EE x_1\in \cup B_1~ \EE x_2 \in \cup B_2 ~ x_1 R x_2 \\
&~\iff~ (\cup B_1)~ R_\clC ~(\cup B_2).\qedhere
\end{align*}
\end{proof}

\begin{theorem}\label{thm:admits-two-classes}
Классы $\clF(m,n,h),~ \clF_*(m,n),~\clG(m,h),~\clG_*(m)$ допускают минимальные фильтрации
для всех $n>m \geq 1,~h\geq 1$. Более того, существуют полиномы $p(x,y,z)$ и $q(x,y)$ такие, что
при всех $n>m \geq 1,~h\geq 1$
класс $\clF(m,n,h)$ допускает $f$-ограниченные минимальные фильтрации для
$$f(x)=\exp^h_2(p(x,h,n-m)),$$
класс $\clG(m,h)$ допускает $f$-ограниченные минимальные фильтрации для
$$f(x)=\exp^h_2(q(x,h)).$$
\end{theorem}
\begin{proof}
Пусть $F\in \clF(m,n,h)$, $\clA_0$ --- конечное разбиение $F$,  $\clA$ --- его измельчение такое, что
$F_\clA\in \clF(m,n,h)$ и мощности сгустков в $F_\clA$ ограничены $(n-m)|\clA_0|$ (лемма \ref{lemma:finite-clusters-part}).
Рассмотрим разбиение $\clA$
$$\clB_0=\{~\{U\in \clA\mid U\subseteq V\}\mid V\in \clA_0\}.$$
Очевидно, $|\clA_0|=|\clB_0|$.
Пусть $\clB$ --- правильное конечное измельчение $\clB_0$ в шкале $F_\clA$, существующее по теореме \ref{thm:finite_depth_part}.
В силу предложений \ref{prop:franz-pmorph} и \ref{prop:twoclasses-defined}, $(F_\clA)_\clB\in \clF(m,n,h)$.
По предложению \ref{prop:filtr-compose}, построенная шкала изоморфна $F_\clC$ для $\clC=\{\cup B\mid B\in \clB\}$.
Осталось заметить, что $\clC$ --- измельчение $\clA_0$.

Если $x$ --- размер начального конечного разбиения $\clA_0$, то в $F_\clA$ все сгустки ограничены $(n-m)x$ и
размер построенной шкалы $(F_\clA)_\clB$ ограничен
$$\exp^h_2\bigl( ((n-m)x+h+1)((n-m)x+\log_2 x)\bigr).$$
%

Для классов $\clG(m,h)$ и $\clG_*(m)$ рассуждение проводится аналогично, с использованием леммы \ref{lemma:finite-clustersMtrans-sat} вместо леммы \ref{lemma:finite-clusters-part}.
Размер построенной шкалы ограничен
\[\exp^h_2\bigl( (x+h+1)(x+\log_2 x)\bigr).\qedhere
\]
\end{proof}

Из этой теоремы следует разрешимость логик классов $\clF(m,n,h)$ и $\clG(m,h)$ для данных
$n>m\geq 1,~ h\geq 1$ за элементарное по Кальмару время (время работы разрешающего алгоритма ограничено некоторой башней экспонент).
\begin{question} Для  $n>m \geq 2,~h\geq 1$,
какова сложность логик классов $\clF(m,n,h)$ и $\clG(m,h)$?
\end{question}

Никакой эффективной оценки на размер фильтрации в классах $\clF_*(m,n)$ и $\clG_*(m)$ доказанная теорема не даёт, и вопрос о разрешимости их логик открыт при $n>m>1$.

\begin{question} Для  $n>m \geq 2$, разрешимы ли логики классов $\clF_*(m,n)$ и $\clG_*(m)$? являются ли они конечно аксиоматизируемыми?
\end{question}
Отрицательный ответ на один из этих вопросов будет означать отсутствие финитной аппроксимируемости логик $\clF(m,n)$ или $\clG(m)$.

Нам представляется вполне возможным, что 
логики классов $\clF_*(m,n)$ и $\clG_*(m)$ как минимум не являются элементарными по Кальмару,
а логики классов $\clF(m,n,h)$ и $\clG(m,h)$  \mbox{$(h-1)$-EXPTIME-трудны}.
\ISH{стоит ли фантазировать?}

\bigskip
Для $n>m>1$, $h\geq 1$, суммируем известные факты и открытые вопросы о финитной аппроксимируемости, разрешимости и конечной аксиоматизируемости:
\begin{center}
\begin{tabular}{|c||c|c|c|}
\hline
&$\clF(m,n,h)$, $\clG(m,h)$&$\clF_*(m,n)$, $\clG_*(m)$& $\clF(m,n)$, $\clG(m)$\\ \hline
{ФА}  &  + &  +& ?\\ \hline
Разрешимость  &  + & ?& ? \\ \hline
КА  &  + & ?& +  \\ \hline
\end{tabular}
\end{center}

\bigskip
Приведем ещё одно следствие полученных результатов.

Формулы языка с $k$ модальностями интерпретируются в шкалах с $k$ отношениями:
$$\mM,w\mo \Di_i \vf  ~\iff ~  \EE v (wR_iv \textrm{ и } M,v\mo \vf).$$

Рассмотрим язык с двумя модальностями. Для класса $\clF$, состоящего из шкал с одним отношением, положим
$$\temp{\clF}:=\{(W,R,R^{-1})\mid (W,R) \in \clF\}.$$

В силу недавних результатов \cite[Theorem 2.7]{KSZ:2014:AiML}, фильтруемость класса шкал $\clF$ влечёт финитную аппроксимируемость не только $\Log\clF$, но и
$\Log \temp{\clF}$.
\begin{corollary} Пусть $n>m \geq 1,~h\geq 1$.
Логики классов $$\temp{\clF(m,n,h)},~ \temp{\clF_*(m,n)},~\temp{\clG(m,h)},~\temp{\clG_*(m)}$$   финитно аппроксимируемы.
\end{corollary}
Известно, что если $\Lvar$ --- каноническая логика, и $\clF$ --- класс всех шкал $\Lvar$, то логика
$\Log\temp{\clF}$ получается из $\Lvar$
добавлением двух аксиом $p\imp \neg \Di_1\neg \Di_2 p$, $p\imp \neg \Di_2\neg \Di_1 p $.
\begin{corollary} При $n>m \geq 1,~h\geq 1$
логики классов $\temp{\clF(m,n,h)},~ \temp{\clG(m,h)}$ разрешимы.
\end{corollary}

Несмотря на установленные теоремы о финитной аппроксимируемости, логики предтранзитивных шкал конечной высоты устроены существенно сложнее, чем соответствующие расширения логики предпорядков $\LS{4}$. Так,
все логики $\LS{4}+\{B_h\}$ локально табличны,
логика
$\LS{4}+\{B_1\}=\LS{5}$ предтаблична.
 Однако ни локальная табличность, ни предтабличность не имеют места уже для $\Log\clG(2,1)$, поскольку у этой логики
имеются неполные по Крипке расширения
\cite{miyazaki2004normal,Kostrzycka:NEXTKTB:2008}
\ISH{Так всё же следует ли верить их результатам?}. Это же верно для всех логик  $\Log\clF(m,n,h)$, $\Log\clG(m,h)$, $n>m\geq 2$, $h\geq 1$, поскольку они включены в $\Log\clG(2,1)$.

\bigskip
\ISH{Что где получено - надо ли указывать? В частности, теоремы \ref{thm:Glivenko} и \ref{thm:compl} были опубликованы в сборнике ИТИС}

\ISH{Благодарности.}

\hide{
Перейти от изучения предтранзитивных логик высоты 1 к случаю логик произвольной конечной высоты  авторам предложил В.Б. Шехтман, за что авторы выражают ему свою признательность. } 

\bibliographystyle{apalike}


\begin{thebibliography}{BCM{\etalchar{+}}03}

\bibitem[Мак75]{Max1975}
Л.Л. Максимова.
\newblock Модальные логики конечных слоев.
\newblock {\em Алгебра и логика}, 14(3):304--319, 1975.

\bibitem[BCM{\etalchar{+}}03]{BaaderDL2003}
Franz Baader, Diego Calvanese, Deborah~L. McGuinness, Daniele Nardi, and
  Peter~F. Patel-Schneider, editors.
\newblock {\em The Description Logic Handbook: Theory, Implementation, and
  Applications}.
\newblock Cambridge University Press, 2003.

\bibitem[BdRV02]{blackburn_modal_2002}
P.~Blackburn, M.~de~Rijke, and Y.~Venema.
\newblock {\em Modal Logic}.
\newblock Cambridge University Press, 2002.

\bibitem[BJ51]{JonssonTarski}
Alfred~Tarski Bjarni~J\'{o}nsson.
\newblock Boolean algebras with operators. part i.
\newblock {\em American Journal of Mathematics}, 73(4):891--939, 1951.

\bibitem[Boo95]{BoolosBook}
George Boolos.
\newblock {\em The Logic of Provability}.
\newblock Cambridge University Press, 1995.

\bibitem[CZ97]{Ch:Za:ML:1997}
Alexander Chagrov and Michael Zakharyaschev.
\newblock {\em Modal Logic}, volume~35 of {\em Oxford Logic Guides}.
\newblock Oxford University Press, 1997.

\bibitem[G\"86]{Goedel33}
Kurt G\"{o}del.
\newblock Eine interpretation des intuitionistischen aussagenkalk\"{u}ls. {\it
  ergebnisse eines mathematischen kolloquiums}, 4:39–40, 1933. english
  translation, with an introductory note by a. s. troelstra.
\newblock In {\em Kurt G\"{o}del: Collected Works. Volume I: Publications
  1929-1936}, pages 296--301. Oxford University Press, 1986.

\bibitem[Gab72]{Gabbay:1972:JPL}
Dov~M. Gabbay.
\newblock A general filtration method for modal logics.
\newblock {\em Journal of Philosophical Logic}, 1(1):29--34, 1972.

\bibitem[GO07]{ModalMogel}
Valentin Goranko and Martin Otto.
\newblock Model theory of modal logic.
\newblock In {\em Handbook of Modal Logic}, pages 249--329. Elsevier, 2007.

\bibitem[GSS09]{ShefSkvGab}
D.~Gabbay, V.~Shehtman, and D.~Skvortsov.
\newblock {\em Quantification in Nonclassical Logic}.
\newblock Elsevier, 2009.

\bibitem[Jan68]{JankovContin68}
V.A. Jankov.
\newblock The construction of a sequence of strongly independent
  superintuitionistic propositional calculi.
\newblock {\em Soviet Mathematics Doklady}, (9):806--807, 1968.

\bibitem[Jan94]{jansana1994}
Ramon Jansana.
\newblock Some logics related to von {W}right's logic of place.
\newblock {\em Notre Dame J. Formal Logic}, 35(1):88--98, 01 1994.

\bibitem[Kos08]{Kostrzycka:NEXTKTB:2008}
Zofia Kostrzycka.
\newblock On non-compact logics in {NEXT(KTB)}.
\newblock {\em Mathematical Logic Quarterly}, 54(6):617--624, 2008.

\bibitem[KS11]{Ku:Shap:TACL:2011}
Andrey Kudinov and Ilya Shapirovsky.
\newblock Finite model property of pretransitive analogs of {S5}.
\newblock In {\em Topology, Algebra and Categories in Logic (TACL 2011)}, pages
  261--264. Marseille, 2011.

\bibitem[KSZ14]{KSZ:2014:AiML}
Stanislav Kikot, Ilya Shapirovsky, and Evgeny Zolin.
\newblock Filtration safe operations on frames.
\newblock In {\em Advances in Modal Logic}, volume~10, pages 333--352. College
  Publications, 2014.

\bibitem[LL32]{Lewis-Langford}
Clarence~Irving Lewis and Cooper~Harold Langford.
\newblock {\em Symbolic logic}.
\newblock Century Co., 1932.

\bibitem[LS66]{LemmonScott1966}
Edward~John Lemmon and Dana Scott.
\newblock {\em Intensional Logics}.
\newblock Stanford, 1966.

\bibitem[Mat55]{Mats-s4s5}
K.~Matsumoto.
\newblock Reduction theorem in {L}ewis's sentential calculi.
\newblock {\em Mathematica Japonicae}, 3:133--135, 1955.

\bibitem[McK41]{McKinsey41}
J.~C.~C. McKinsey.
\newblock A solution of the decision problem for the lewis systems s2 and s4,
  with an application to topology.
\newblock {\em The Journal of Symbolic Logic}, 6(4):117--134, 1941.

\bibitem[Miy04]{miyazaki2004normal}
Yutaka Miyazaki.
\newblock Normal modal logics containing {KTB} with some finiteness conditions.
\newblock {\em AiML-2004: Advances in Modal Logic}, page 264, 2004.

\bibitem[MT48]{McKinsey-Tarski48}
J.~C.~C. McKinsey and Alfred Tarski.
\newblock Some theorems about the sentential calculi of lewis and heyting.
\newblock {\em Journal of Symbolic Logic}, 13(1):1--15, 1948.

\bibitem[Seg68]{Segerberg1968}
Krister Segerberg.
\newblock Decidability of four modal logics.
\newblock {\em Theoria}, 34:21--25, 1968.

\bibitem[Seg71]{Seg_Essay}
Krister Segerberg.
\newblock {\em An Essay in Classical Modal Logic, Volume 3}.
\newblock Filosofska Studier, vol.13. Uppsala Universitet, 1971.

\bibitem[Seg73]{Franz-Bull}
Krister Segerberg.
\newblock Franzen's proof of {B}ull's theorem.
\newblock {\em Ajatus}, 35:216--221, 1973.

\bibitem[SF96]{SmullFitt}
Raymond~M. Smullyan and Melvin Fitting.
\newblock {\em Set Theory and the Continuum Problem}.
\newblock Oxford University Press, 1996.

\bibitem[She04]{shehtman2004filtration}
Valentin~B. Shehtman.
\newblock Filtration via bisimulation.
\newblock {\em Advances in Modal Logic}, 5:289--308, 2004.

\bibitem[Sol76]{Solovay76}
Robert~M. Solovay.
\newblock Provability interpretations of modal logic.
\newblock {\em Israel Journal of Mathematics}, 25(3-4):287--304, 1976.

\bibitem[Var07]{Vardi06automata}
Moshe~Y. Vardi.
\newblock Automata-theoretic techniques for temporal reasoning.
\newblock In {\em In Handbook of Modal Logic}, pages 971--989. Elsevier, 2007.

\bibitem[Ven95]{venema-cyl1995}
Yde Venema.
\newblock Cylindric modal logic.
\newblock {\em Journal of Symbolic Logic}, 60(2):591--623, 1995.

\bibitem[WZ07]{ModalDecWZ}
Frank Wolter and Michael Zakharyaschev.
\newblock Modal decision problems.
\newblock In {\em Handbook of Modal Logic}, pages 42--489. Elsevier, 2007.

\end{thebibliography}

\newcommand{\etalchar}[1]{$^{#1}$}

\end{document}